\documentclass[11pt, leqno]{article}
\usepackage{amsmath,amssymb,latexsym,amsthm,mathrsfs}
\usepackage{authblk}
\usepackage[paper=a4paper,left=20mm,right=20mm,top=20mm,bottom=20mm]{geometry}
\usepackage{amsmath}
\usepackage{amssymb}
\usepackage{amsthm}
\usepackage{array} 
\usepackage{calc,ifthen}
\usepackage{graphicx} 
\usepackage{color,transparent}
\usepackage{tikz}
\usepackage{pgfplots}
\usepackage{multicol}
\usepackage{braids}
\usepackage{wrapfig}
\usepackage{float}
\usepackage{verbatim}
\allowdisplaybreaks
\usepackage[sorting=none, backend=bibtex]{biblatex}
\graphicspath{ {./Images/} }

\theoremstyle{plain} 
\newtheorem{defn}{Definition}[section]
\newtheorem*{defn*}{Definition}

\newtheorem*{lem*}{Lemma}
\newtheorem{lem}[defn]{Lemma}
\newtheorem{thm}{Theorem}[section]
\newtheorem*{thm*}{Theorem}
\makeatletter
\def\thmhead@plain#1#2#3{
  \thmname{#1}\thmnumber{\@ifnotempty{#1}{ }\@upn{#2}}
  \thmnote{ {\the\thm@notefont#3}}}
\let\thmhead\thmhead@plain
\makeatother

\newtheorem*{cor*}{Corollary}

\newtheorem*{pro*}{Property}
\newtheorem{prop}[defn]{Proposition}
\newtheorem*{prop*}{Proposition}
\newtheorem*{claim*}{Claim}

\newtheorem{con}[defn]{Conjecture}

\theoremstyle{definition} 

\newtheorem*{rem*}{Remark}
\newtheorem*{rems*}{Remarks}

\newtheorem*{prob*}{Problem}

\newtheorem*{nota*}{Notation}
\newtheorem*{ques*}{Questions}
\title{Gordian Adjacency for Positive Braid Knots}

\author[1]{Tolson H. Bell \thanks{
                    tbell37@gatech.edu}}
\author[2]{David C. Luo \thanks{
                    david.luo@emory.edu}}
\author[3]{Luke Seaton \thanks{
                    lukeseaton98@gmail.com}}
\author[4]{Samuel P. Serra \thanks{
                    samuel.serra@colorado.edu}}
\affil[1]{School of Mathematics, Georgia Institute of Technology}
\affil[2]{Department of Mathematics, Emory University}
\affil[3]{Department of Mathematics and Statistics, Louisiana Tech University}
\affil[4]{Department of Mathematics, University of Colorado Boulder}

\begin{document}
\date{}
\maketitle

\begin{abstract}

A knot $K_1$ is said to be \textit{Gordian adjacent} to a knot $K_2$ if $K_1$ is an intermediate knot on an unknotting sequence of $K_2$. We extend previous results on Gordian adjacency by showing sufficient conditions for Gordian adjacency between classes of positive braid knots through manipulations of braid words. In addition, we explore unknotting sequences of positive braid knots and give a proof that there are only finitely many positive braid knots for a given unknotting number.

\end{abstract}

\section{Introduction}

A \textit{mathematical knot} is a non-intersecting embedding of a closed curve in three-dimensional space that is considered unique only up to \textit{isotopy}, a continuous deformation of this embedding without passing the curve through itself. We can visualize a knot by projecting the embedding onto a two-dimensional plane to form a \textit{knot diagram}. A knot diagram contains \textit{crossings}, where parts of the curve intersect in the projection. When we do pass one part of the curve through another, we have made a \textit{crossing change}.
\vspace{-0.2in}
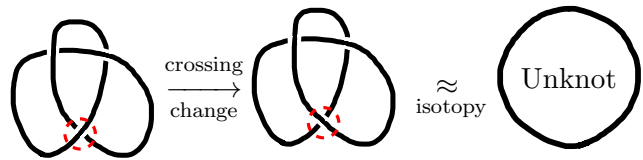
\begin{figure}[H]
\begin{center}
	\begin{tikzpicture}
		\node at (0,1.2){};
		\node at (0,0) {\includegraphics[width=.75in]{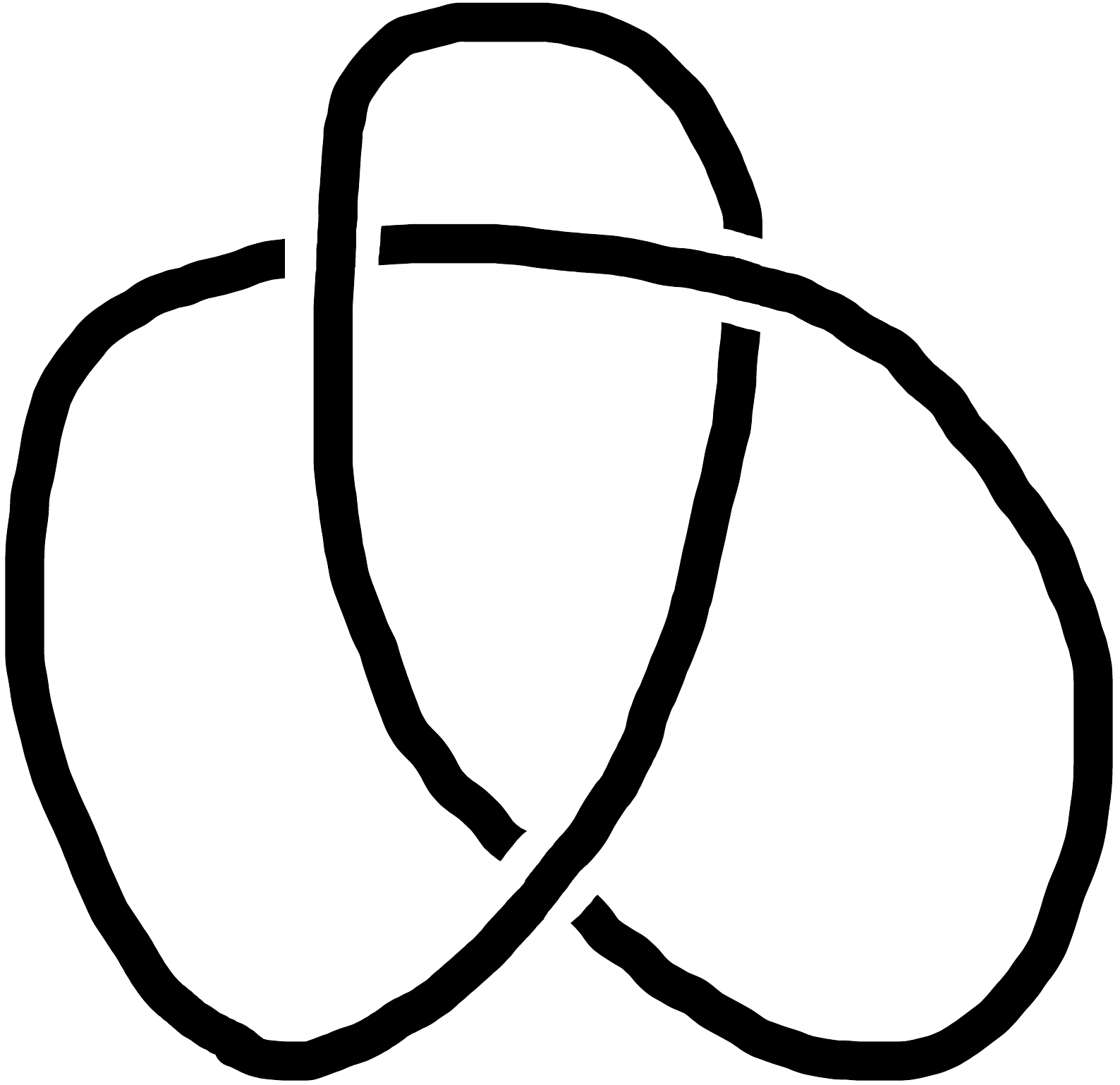}};

		\draw[very thick, red, dashed] (-.03,-.6) circle (0.2cm);
		\node at (1.6,0) {$\overset{\text{crossing}}{\underset{\text{change}}{\xrightarrow{\hspace{2em}}}}$};
  \end{tikzpicture}
  \begin{tikzpicture}
    \useasboundingbox (0,0) rectangle (2,2);
		\node at (.8,1.2) {\includegraphics[width=.75in]{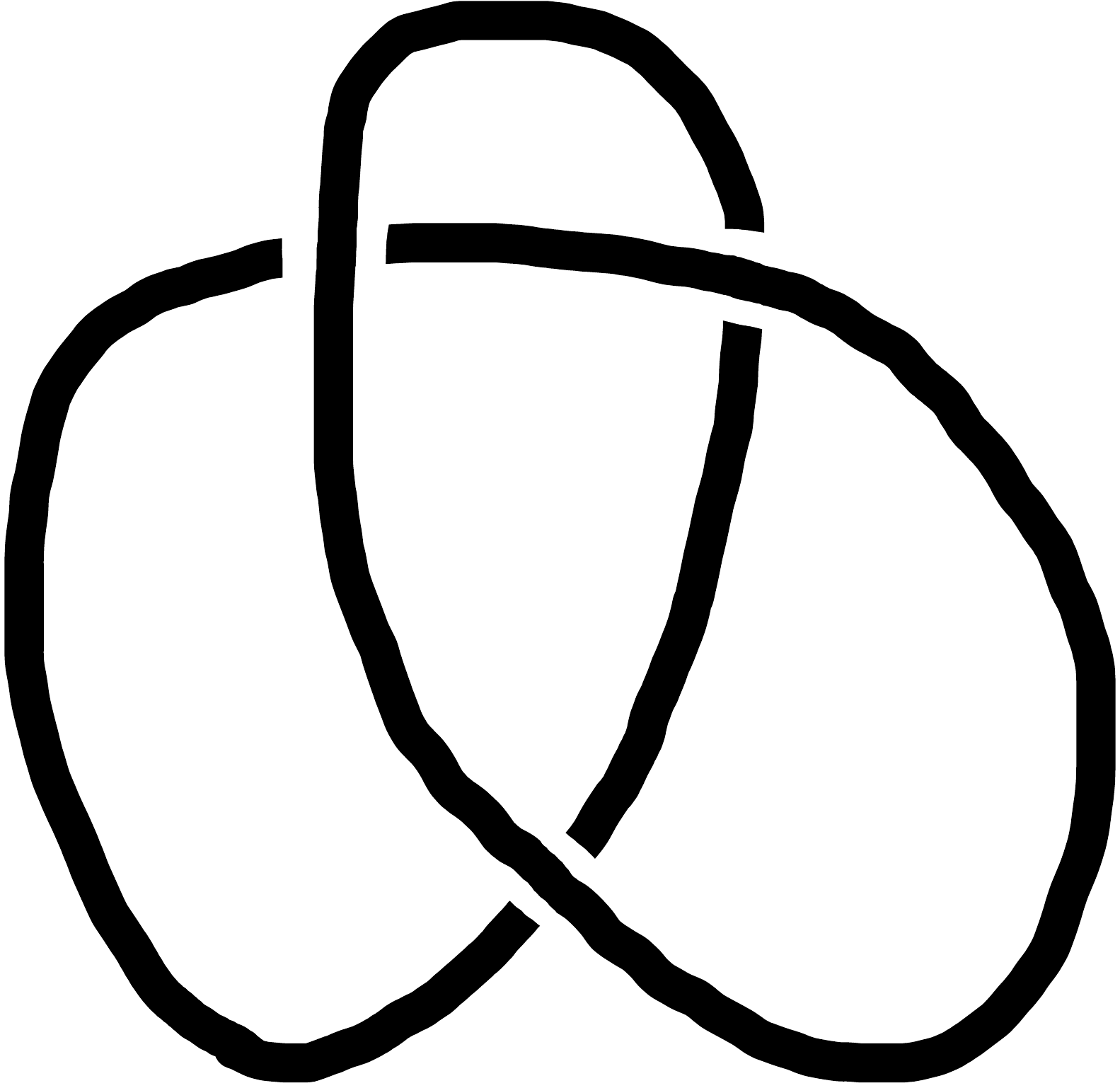}};
		\draw[very thick, red, dashed] (0.8,.6) circle (0.2cm);

  \end{tikzpicture}
  \begin{tikzpicture}
    \useasboundingbox (0,0) rectangle (3,3);
		\node at (.3,.9) {$\underset{\text{ isotopy}}{\approx}$};
		\node at (1.9,1.2) {\includegraphics[width=.75in]{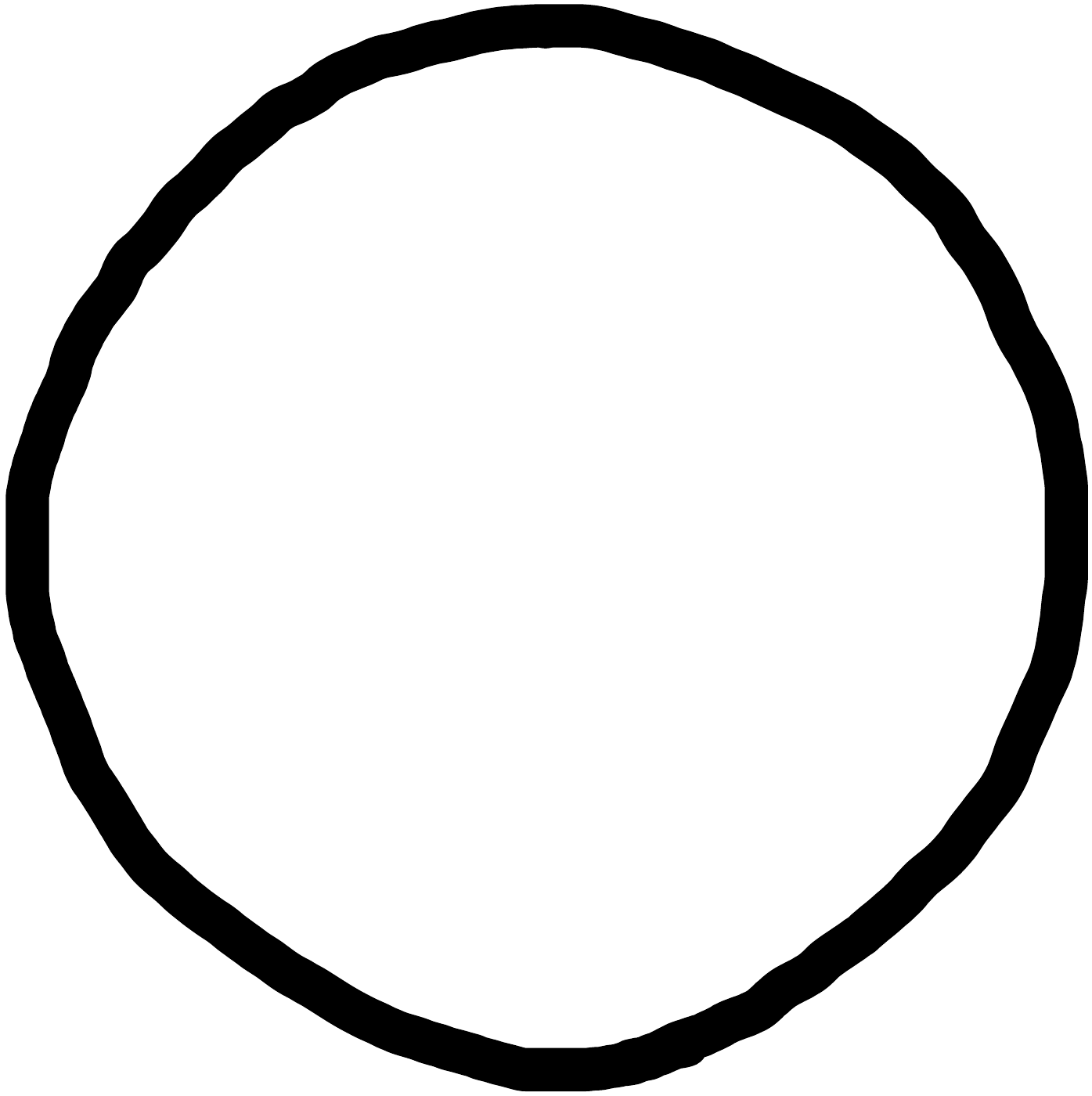}};
		\node at (1.9,1.2){Unknot};
	\end{tikzpicture}
\end{center}
	\caption{Crossing Change and Isotopy of the Left-handed Trefoil Knot}
	\label{tref}
\end{figure}

From these definitions, we obtain the concepts of \textit{unknotting numbers} and \textit{Gordian distance} between knots \cite{KnotBook}.
\begin{defn}
The Gordian distance $d_g(K_1,K_2)$ between two knots $K_1$ and $K_2$ is the minimal number of crossing changes needed to change $K_2$ into $K_1$.
\end{defn}
\begin{defn}\label{unknottingno}
The unknotting number $u(K)$ of a knot $K$ is given by $d_g(O,K)$, where $O$ is the unknot.
\end{defn}
Figure \ref{tref} shows that the unknotting number of the trefoil knot is no more than one, as it demonstrates how the trefoil can be made into the unknot with only one crossing change.

An important class of knots we study in great detail is torus knots. A \textit{torus knot} is a closed curve on the surface of an unknotted torus that does not intersect itself anywhere \cite{KnotBook}. We denote a torus knot by $T(p,q)$, where $p$ and $q$ denote the number of longitudinal and meridional twists around the torus respectively such that $p$ is coprime to $q$. Note that $T(p,q)$ is isotopic to $T(q,p)$, that is, $T(p,q)=T(q,p)$ \cite{KnotBook}.

\begin{figure}[h]
\begin{center}
    \centering
    \includegraphics[scale=.41]{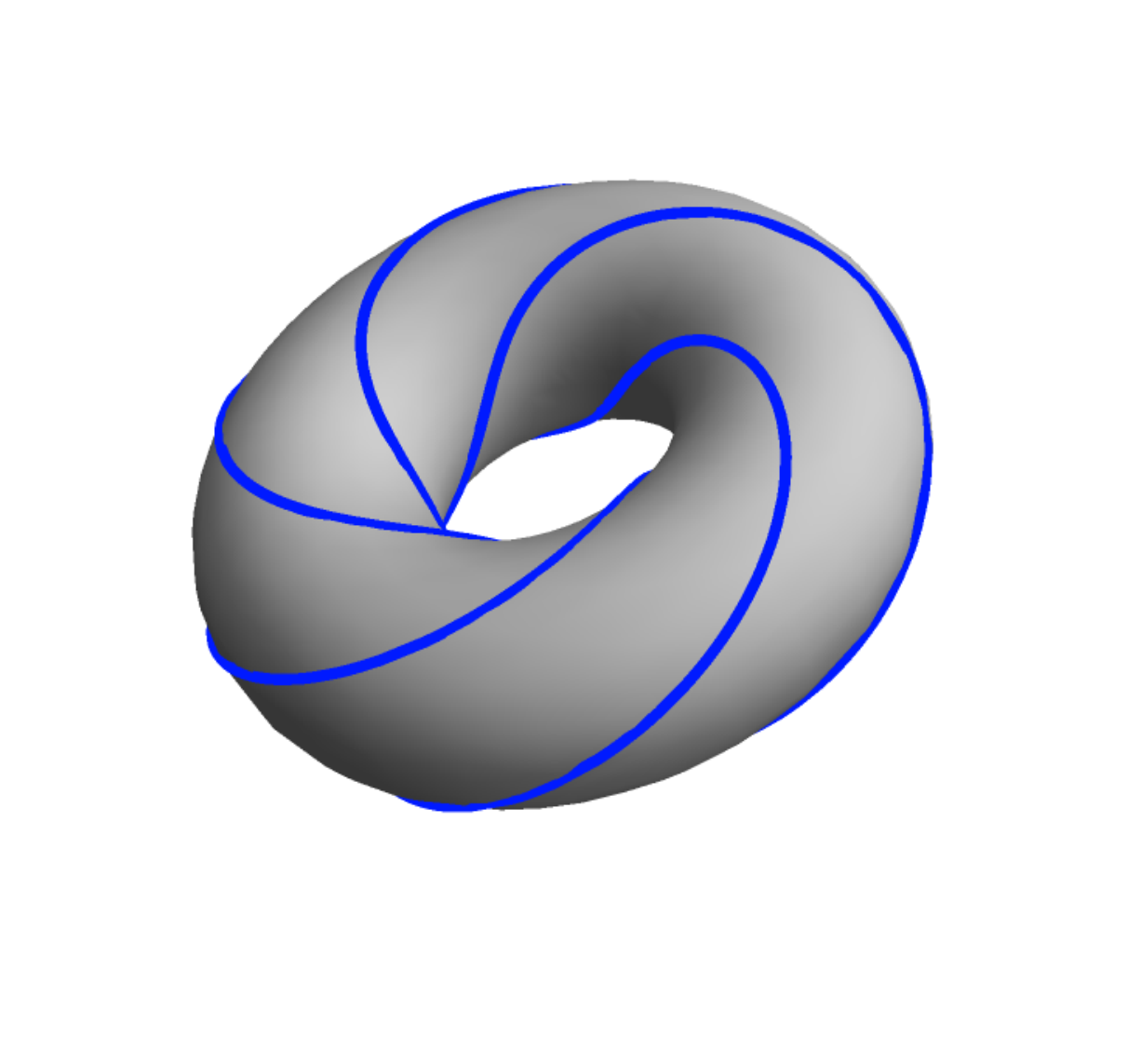}
\label{torus:}
    \centering
    \includegraphics[scale=.25]{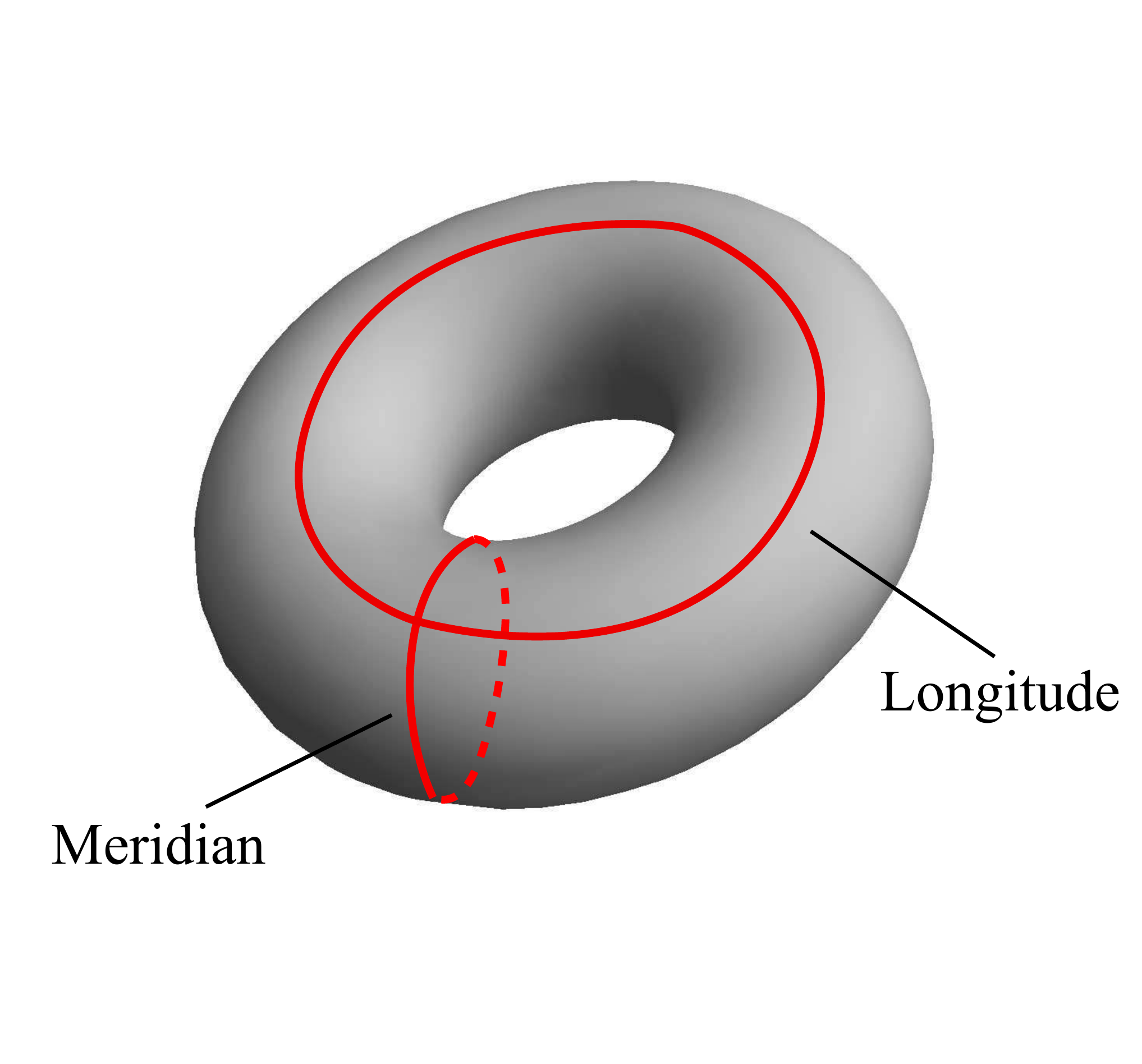}
\end{center}
\caption{T(2,5) and the Axes of the Torus}
\end{figure}

In 2014, Peter Feller introduced the concept of \textit{Gordian adjacency} between knots \cite{Feller}, an idea closely related to the description of \textit{unknotting sequences}.

\begin{defn}[(Feller Definition 1)]
A knot $K_1$ is said to be Gordian adjacent to a knot $K_2$, $K_1 \leq_g K_2$, if $d_g(K_1,K_2)=u(K_2)-u(K_1)$.
\end{defn}

\begin{defn}
An unknotting sequence of a knot $K$ is a series of knots beginning with $K$ and ending with the unknot $O$ such that for any two consecutive knots $K_2$ followed by $K_1$ in the series, $u(K_2)-u(K_1)=1$ and $d_g(K_1,K_2)=1$. 
\end{defn}
Note that in our definition of an unknotting sequence, it is optimal in the sense that it is the fewest number of knots needed to get from a particular knot to the unknot.

In general, the Gordian distance is difficult to compute. However, if $K_1$ and $K_2$ are Gordian adjacent, then their Gordian distance is the difference of their unknotting numbers. An equivalent definition for Gordian adjacency between knots $K_1$ and $K_2$ is that $K_1 \leq_g K_2$ if $K_1$ appears in an unknotting sequence of $K_2$ \cite{Feller}.  

Feller was able to show that under certain circumstances, Gordian adjacencies are guaranteed between certain classes of torus knots \cite{Feller}.

\begin{thm}[(Feller Theorem 2)]\label{fel2}
If $(n, m)$ and $(a, b)$ are pairs of coprime positive integers with $n\leq a$ and $m \leq b$, then  $T(n,m) \leq_g T(a,b)$.
\end{thm}

\begin{thm}[(Feller Theorem 3)]\label{fel3}
If $n$ and $m$ are positive integers with $n$ odd and $m$ not a multiple
of 3, then  $T(2,n) \leq_g T(3,m)$ if and only if $n \leq \frac{4}{3}m + \frac{1}{3}$.
\end{thm}

Our main goal in this paper is to build off Feller's results by finding Gordian adjacencies among other classes of positive braid knots. In Section 3, we describe new techniques which help prove our results by representing positive braid knots by their braid words. In Section 4, we present our results on Gordian adjacency for certain classes of positive braid knots. Listed below are the major theorems we obtain.

\begin{thm*}[\ref{ci}]\label{Theorem1}
If $n$ and $k$ are positive integers, then  $T(n,n^2k+1) \leq_g T(n+1,(n^2-1)k+1)$.
\end{thm*}

\begin{thm*}[\ref{cin}]\label{Theorem2}
If $n$ and $k$ are positive integers, then  $T(n,n^2k+n+1) \leq_g T(n+1,(n^2-1)k+n)$.
\end{thm*}

\begin{thm*}[\ref{thrfou}]\label{Theorem3}
If $a$ and $b$ are positive integers, then  $T(3,a) \leq_g T(4,b)$ if $a\leq \frac{9b+5}{8}$.
\end{thm*}

\begin{thm*}[\ref{twofou}]
If $a$ and $b$ are positive integers, then  $T(2,a) \leq_g T(4,b)$ if $a\leq \frac{3b+3}{2}$.
\end{thm*}

\begin{thm*}[\ref{w}]\label{Theorem5}
Let $\beta$ in $B_n$ be a positive braid where $\beta=\beta'w$ and $\hat\beta$ and $\hat\beta'$ are knots. If $\hat w$ is a link with $n$ components, then $\hat\beta' \leq_g \hat\beta$.
\end{thm*}

Using our new techniques, we present a proof in Section 5 showing that every positive braid knot can be unknotted in a way such that every intermediate knot is a positive braid knot. We also prove that every positive braid knot has only a finite number of these positive unknotting sequences.

\begin{thm*}[\ref{posunk}]\label{Theorem6}
Every positive braid knot has an unknotting sequence that consists of only positive braid knots.
\end{thm*}

\begin{thm*}[\ref{finite}]\label{Theorem7}
For every positive integer $m$, there exist a finite number of positive braid knots with unknotting number $m$.
\end{thm*}

\section*{Acknowledgements}
This research was conducted as part of the 2018 SURIEM REU program at Lyman Briggs College of Michigan State University, under the supervision of Dr.~Robert Bell. We gratefully acknowledge support from the National Security Agency (NSA Award No. H98230-18-1-0042), the National Science Foundation (NSF Award No. 1559776), and Michigan State University. We thank our mentors Dr.~Katherine Raoux and David Storey for their guidance throughout this project. In addition, we are grateful to Dr.~Peter Feller for his advice and Eric Zhu for coding support. Finally, we would like to thank the editors of the \textit{Rose-Hulman Undergraduate Mathematics Journal} and the anonymous referee for insightful feedback. This paper was inspired by Feller \cite{Feller}.

\section{Preliminaries: Braids and Torus Knots}
Throughout this paper, we refer to knots by their braid words. In this section, we present the braid group and describe how to represent positive braid knots as braids.

A braid $\beta$ is defined as a set of $n$ strands which begin on a horizontal bar and end on a lower horizontal bar. Each strand may only intersect any horizontal plane once, which allows the strands to cross each other in the specific ways described later \cite{KnotBook}.

The \textit{closure} of a braid, $\hat\beta$, is formed by attaching the top and bottom bar such that the beginning of each strand connects to the end of a strand, forming a \textit{link}, one or more knots \cite{StudyOfBraids}. Every knot can be represented as the closure of some braid \cite{Alexander}.

To represent braids algebraically, we refer to their generators $\sigma_i$'s which represent the $i+1^{st}$ strand from the left crossing over the $i^{th}$ strand (read from top to bottom). If the $i+1^{st}$ strand from the left instead crosses under the $i^{th}$ strand, we denote this by $\sigma^{-1}_i$.

\begin{figure}[H]
\begin{center}
	\begin{center}
		\begin{tikzpicture}
			\braid[number of strands=4, width=.5cm,height=0.3cm,line width=2pt,style strands={1,3}{gray}, style strands={2,4}{blue}] (s1) s_1^{-1};
			\node at (1.25,-1.25){$\sigma_{1}$};
		\end{tikzpicture}
	\hspace{1cm}
		\begin{tikzpicture}
			\braid[number of strands=4, width=.5cm,height=0.3cm,line width=2pt,style strands={1,3}{gray}, style strands={2,4}{blue}] (s2) s_2^{-1};
			(\node at (1.25,-1.25){$\sigma_{2}$};
		\end{tikzpicture}
	\hspace{1cm}
		\begin{tikzpicture}
			\braid[number of strands=4, width=.5cm,height=0.3cm,line width=2pt,style strands={1,3}{gray}, style strands={2,4}{blue}] (s3) s_3^{-1};
			\node at (1.25,-1.25){$\sigma_{3}$};
		\end{tikzpicture}
	\end{center}
\vspace{.5cm}
		\begin{tikzpicture}
			\braid[number of strands=4, width=.5cm,height=0.3cm,line width=2pt,style strands={2,3}{gray}, style strands={1,4}{blue}] (s1) s_1;
			\node at (1.25,-1.25){$\sigma_{1}^{-1}$};
		\end{tikzpicture}
	\hspace{1cm}
		\begin{tikzpicture}
			\braid[number of strands=4, width=.5cm,height=0.3cm,line width=2pt,style strands={1,2}{gray}, style strands={3,4}{blue}] (s2) s_2;
			
			\node at (1.25,-1.25){$\sigma_{2}^{-1}$};
		\end{tikzpicture}
	\hspace{1cm}
		\begin{tikzpicture}
			\braid[number of strands=4, width=.5cm,height=0.3cm,line width=2pt,style strands={1,4}{gray}, style strands={2,3}{blue}] (s3) s_3;
			\node at (1.25,-1.25){$\sigma_{3}^{-1}$};
		\end{tikzpicture}
\end{center}
\caption{The Generators of $B_4$ and Their Inverses}
\label{gens}
\end{figure}
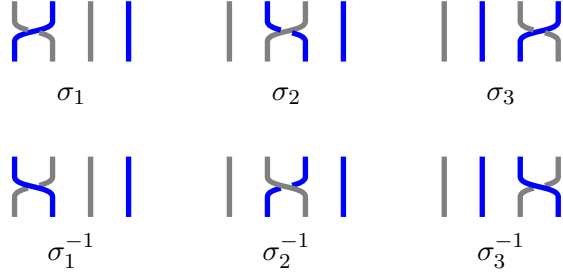

We call $\sigma^{-1}_i$ the inverse of $\sigma_i$ as it is isotopic to the identity element when concatenated with $\sigma_i$. Hence we have that $\sigma_1,\dotsc,\sigma_{n-1}$ are the generators of the braid group $B_n$, which form a group under the operation of concatenation. Note that changing a generator to its inverse corresponds to a crossing change in the knot which the braid represents. We have two relations on the braid group $B_n$, also known as braid isotopies:

\begin{itemize}
\item for any $1 \leq i \leq n-2$, $\sigma_i \sigma_{i+1} \sigma_i = \sigma_{i+1} \sigma_i \sigma_{i+1} $,
\item for any $1 \leq i \leq n-3$ and $i+2 \leq j \leq n-1$, $\sigma_j \sigma_i = \sigma_i \sigma_j$.
\end{itemize}

Both of these braid relations produce isotopic braid closures \cite{Artin}. Given a knot $K$, the \textit{braid index} \normalfont $brd(K)$ of $K$ is the smallest positive integer $n$ such that there exists a braid $\beta$ in $B_n$ where $\hat\beta = K$.

We refer to a positive generator of $B_n$ as a \textit{positive crossing}. A \textit{positive braid} is a braid in which all crossings are positive (we have no inverse generators). A \textit{positive braid knot} is a knot which is isotopic to the closure of some positive braid. While unknotting numbers can be hard to compute in general, we have that for positive braid knots, the unknotting number of $\hat{\beta}$ which is an integer quantity, is given by 
$$u(\hat{\beta})= \frac{\ell-n+1}{2}$$ where $\ell$ 
denotes the length of the braid word (the number of generators it contains) and $n$ the number of strands on which the braid is expressed. This was proven as a lower bound in 1984 \cite{FrenchBook} and as an upper bound in 2004 \cite{Livingston}. 

In our study of torus knots, we usually represent torus knots by using braids that they are the closure of. The torus knot $T(p,q)$ can be represented by the braid word $\sigma_{p-1}\cdots\sigma_1$ repeated $q$ times, which we denote as $(\sigma_{p-1}\cdots\sigma_1)^q$ \cite{KnotBook}.

\begin{figure}[H]
\begin{center}
\begin{tikzpicture}
\braid[width=.5cm,height=0.3cm,line width=2pt,style strands={1,3}{gray}, style strands={2,4}{blue}] s_3^{-1} s_2^{-1} s_1^{-1} s_3^{-1} s_2^{-1} s_1^{-1} s_3^{-1} s_2^{-1} s_1^{-1} s_3^{-1} s_2^{-1} s_1^{-1} s_3^{-1} s_2^{-1} s_1^{-1} ;
\end{tikzpicture}
\end{center}
\caption{T(4,5) as a Braid in $B_4$}
\end{figure}
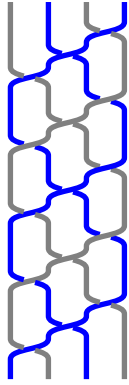

By flipping this braid representation over a vertical axis (viewing it from behind), we see that $(\sigma_1\cdots\sigma_{p-1})^q$ is also a valid braid representation of $T(p,q)$. Since torus knots are a particular class of positive braid knots, the braid representation of torus knots tells us that $$u\left(T(p,q)\right)=\frac{(p-1)(q-1)}{2}.$$

\section{Rules for Constructing Gordian Adjacencies}
Once we have represented two knots as positive braids, we wish to manipulate these braids in order to uncover Gordian adjacencies between the original knots. We have reduced these manipulations to the application of five rules. Figure \ref{fig:fiverules} shows a graphical representation of these five rules, which we summarize in the following paragraph.

The Distant Generators Rule and Neighboring Generators Rule are relations on the braid group and therefore do not change the braid closure. The Conjugation Rule also does not change the braid closure. These three rules are bidirectional, while the latter two rules are not. The Markov Destabilization Rule similarly does not change the closure of the braid by Markov's Theorem \cite{Markov}. The Crossing Change Rule is the only rule that changes the closure of a braid up to knot isotopy. It arises from performing a crossing change which switches a generator to its inverse and canceling out the resulting pair. 

Throughout this paper, we use an equals sign to denote that the closure of our new braid word is isotopic to the closure of the braid word preceding it and an arrow to denote when we have made one or more crossing changes via the Crossing Change Rule. We note that none of these five rules change the number of components in the closure of our braid word.

\begin{figure}[H]
\begin{center}
\begin{multicols}{2}
\begin{center}
\begin{tikzpicture}
    \braid[height=15pt, width = 15pt,
    style all floors={fill=blue!30!, fill opacity=.35},
    style floors={3}{dashed,fill=gray!30!, fill opacity=.35},
    floor command={
      \fill (\floorsx,\floorsy) rectangle (\floorex,\floorey);
    }, line width=2pt](braid) at (2,0)  a_2^{-1}  | a_1^{-1} |  a_3^{-1}  a_2^{-1} ;
  \end{tikzpicture}
  \hspace{1em}
\begin{tikzpicture}
  \useasboundingbox (0,0) rectangle (0,4);
  \node at (0,1.3) {$=$};
\end{tikzpicture}
\hspace{1em}
  \begin{tikzpicture}
    \braid[height=15pt, width = 15pt,
    style all floors={fill=gray!30!, fill opacity=.35},
    style floors={3}{dashed,fill=blue!30!, fill opacity=.35},
    floor command={
      \fill (\floorsx,\floorsy) rectangle (\floorex,\floorey);
    }, line width=2pt](braid) at (2,0)  a_2^{-1} | a_3^{-1} | a_1^{-1}  a_2^{-1} ;
  \end{tikzpicture}\\
\textbf{Distant Generators Rule:} If $|i-j|>1$, then $\sigma_i\sigma_j=\sigma_j\sigma_i$.\\
\end{center}

\begin{center}
\begin{tikzpicture}
  \braid[height=20pt, width = 20pt,line width=2pt, style strands={1,2}{gray}, style   strands={3}{blue}](braid) a_1^{-1} a_2^{-1} a_1^{-1} ;
\end{tikzpicture}
\hspace{1em}
\begin{tikzpicture}
  \useasboundingbox (0,0) rectangle (0,4);
  \node at (0,1.3) {$=$};
\end{tikzpicture}
\hspace{1em}
\begin{tikzpicture}
  \braid[height=20pt, width = 20pt, line width=2pt, style strands={1,2}{gray}, style   strands={3}{blue}](braid) a_2^{-1} a_1^{-1} a_2^{-1};
\end{tikzpicture}\\
\textbf{Neighboring Generators Rule:} For any $1\le i\le n-2$, $\sigma_i\sigma_{i+1}\sigma_i =\sigma_{i+1}\sigma_i\sigma_{i+1}$.\\
\end{center}
\end{multicols}
\end{center}

\begin{center}
\begin{multicols}{2}
\begin{center}
\includegraphics[height=97pt]{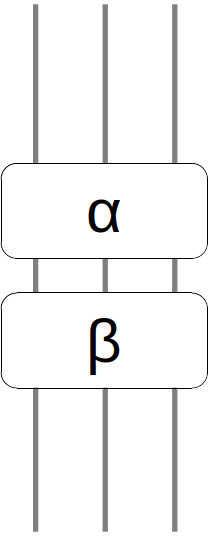}
\hspace{1em}
\begin{tikzpicture}
  \useasboundingbox (0,-2) rectangle (0,2);
  \node at (0,0) {$=$};
\end{tikzpicture}
\hspace{1em}
\includegraphics[height=97pt]{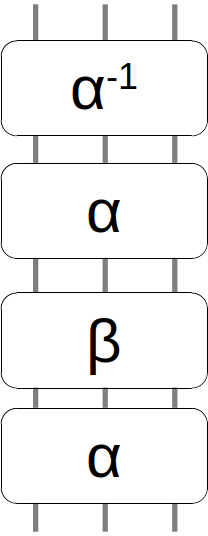}
\hspace{1em}
\begin{tikzpicture}
  \useasboundingbox (0,-2) rectangle (0,2);
  \node at (0,0) {$=$};
\end{tikzpicture}
\hspace{1em}
\includegraphics[height=97pt]{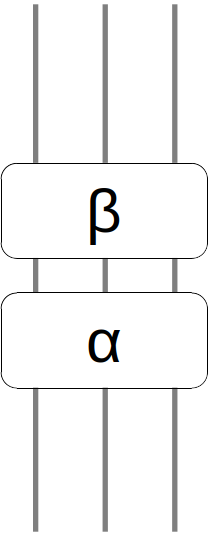}\\
\textbf{Conjugation Rule:} The closure of the braid $\alpha\beta$ is isotopic to the closure of $\beta\alpha$ via a conjugation by $\alpha$.
\end{center}
\begin{center}
\begin{tikzpicture}
  \braid[height=15pt, width = 15pt, line width=2pt,blue, style  strands={3}{gray}, style strands={1,2}{draw=none}](braid) a_3^{-1} a_3^{-1} a_4^{-1} a_3^{-1} a_3^{-1} ;
\end{tikzpicture}
\hspace{1em}
\begin{tikzpicture}
  \useasboundingbox (0,0) rectangle (0,4);
  \node at (0,1.5) {$=$};
\end{tikzpicture}
\hspace{1em}
\begin{tikzpicture}
  \braid[height=15pt, width = 15pt, line width=2pt,blue, style  strands={1}{gray},style strands={3,4}{draw=none}](braid) a_1^{-1} a_1^{-1}  a_3 a_1^{-1} a_1^{-1} ;
\end{tikzpicture}\\
\textbf{Markov Destabilization Rule:} If the braid word contains $\sigma_n$ and $i<n$ for every other $\sigma_i$ in the braid word, then $\sigma_n$ can be deleted.
\end{center}

\end{multicols}
\end{center}

\begin{center}
\begin{tikzpicture}
  \braid[height=15pt, width = 15pt,
  style all floors={fill=yellow},
  floor command={
  \fill (\floorsx,\floorsy) rectangle (\floorex,\floorey);
  }, line width=2pt, number of strands=3, style strands={3,5}{gray}, style   strands={4}{blue}, style strands={1,2}{draw=none}](braid) a_4^{-1} a_3^{-1} a_4^{-1} a_4^{-1} a_3^{-1};
\end{tikzpicture}
\hspace{.75em}
\begin{tikzpicture}
  \useasboundingbox (0,0) rectangle (0,4);
  \node at (0,1.5) {$\rightarrow$};
\end{tikzpicture}
\hspace{.75em}
\begin{tikzpicture}
  \braid[height=15pt, width = 15pt,
  style all floors={fill=yellow},
  floor command={
  \fill (\floorsx,\floorsy) rectangle (\floorex,\floorey);
  }, line width=2pt, number of strands=3, style strands={1,3}{gray}, style   strands={2}{blue}](braid) a_2^{-1} a_1^{-1} a_2^{-1} a_2 a_1^{-1};
\end{tikzpicture}
\hspace{.75em}
\begin{tikzpicture}
  \useasboundingbox (0,0) rectangle (0,4);
  \node at (0,1.5) {$=$};
\end{tikzpicture}
\hspace{.75em}
\begin{tikzpicture}
  \braid[height=15pt, width = 15pt, line width=2pt, number of strands=5, style strands={1,3}{gray}, style strands={2}{blue}, style strands={4,5}{draw=none}](braid) a_2^{-1} a_1^{-1} a_4 a_4 a_1^{-1};
\end{tikzpicture}\\
\textbf{Crossing Change Rule:} If the braid word contains\\$\sigma_i\sigma_i$, then these two letters can be deleted.\\
\end{center}

\caption{Examples and Descriptions of the Five Rules}
\label{fig:fiverules}
\end{figure}

\begin{lem}\label{one}
If we start with a positive braid word $\beta$ whose braid closure is a knot and, by the successive application of the five rules, change $\beta$ into the empty braid word, then the sequence of knots corresponding to the closures of all braid words in the process (up to isotopy) forms an unknotting sequence of $\hat\beta$.
\end{lem}
\begin{proof}
Since none of the five rules change the number of components in our braid closure, this sequence ends with the unknot, the closure of the empty braid word on one strand. The only time we change our braid closure beyond isotopy is when we use the Crossing Change Rule which creates exactly one crossing change. As this is a sequence that starts from a knot and ends with the unknot, to show that this is an unknotting sequence it is enough to prove that every crossing change decreases the unknotting number by one. Recall that the unknotting number of a positive braid knot is given by $\frac{\ell-n+1}{2}$ where $n$ is the number of strands on which the braid knot is expressed and $\ell$ is the length of the braid word \cite{Livingston}. An application of the Crossing Change Rule does not change $n$ but decreases $\ell$ by two. Therefore, it also decreases the unknotting number by one.
\end{proof}

\begin{lem}
\label{AA}
Let $\beta'$ be a subword of a positive braid word $\beta$ where $i\le n$ for all $\sigma_i$ in $\beta'$. Using the five rules, $\beta$ can be made into a braid word that replaces $\beta'$ with a subword containing no more than one $\sigma_n$.
\end{lem}
\begin{proof}
We proceed by induction on $\ell'$, the length of $\beta'$. Our base case is $\ell'=0$, when clearly $\beta'$ is already in the required form for any $n$. Note that our inductive hypothesis requires that the replacement for $\beta'$ does not have any generators with index higher than $n$ and is not longer than $\beta'$ was.

Now, given some $\beta'$ with $\ell'>0$, let $\beta''$ be the subword that includes all of $\beta'$ except for the first generator of $\beta'$. Because $\ell''<\ell'$, we can use the inductive hypothesis with the same value of $n$ to replace $\beta''$ with a subword that contains at most one $\sigma_n$. Now, if the first generator of $\beta'$ is not $\sigma_n$ or if the replacement for $\beta''$ has no $\sigma_n$, we are done, as $\beta'$ is already in the required form.

Otherwise, we have exactly two $\sigma_n$. Let $\gamma$ be the subword between the two $\sigma_n$, not including either $\sigma_n$. As the length of $\gamma$ is less than the original length of $\beta'$, we can use the inductive hypothesis on $\gamma$ with $n-1$. Now, the replacement for $\gamma$ has either one or zero $\sigma_{n-1}$, and no generator with index higher than $n-1$.

If the replacement for $\gamma$ has no $\sigma_{n-1}$, we can then move the second $\sigma_n$ right next to the first using the Distant Generators Rule. Then, we can delete the two $\sigma_n$ via the Crossing Change Rule. This leaves $\beta'$ with no $\sigma_n$, so we are done.

If the replacement for $\gamma$ has one $\sigma_{n-1}$, we can move the two $\sigma_n$ to either side of that $\sigma_{n-1}$ using the Distant Generators Rule. We then use the Neighboring Generators Rule to make $\sigma_n\sigma_{n-1}\sigma_n$ into $\sigma_{n-1}\sigma_n\sigma_{n-1}$. This leaves $\beta'$ with exactly one $\sigma_n$, so we are done.
\end{proof}

The preceding algorithm does not use the Markov Destabilization Rule or the Conjugation Rule, meaning it can be performed locally on any subword of a braid word and still follow the five rules on the whole braid word.

\begin{lem}
\label{two}
Every finite positive braid word can be made into the empty braid word using the five rules.
\end{lem}
\begin{proof}
Let $\sigma_n$ be the letter with the highest subscript in the braid word. By Lemma \ref{AA}, we can make our braid word into a braid word with no more than one $\sigma_n$. If this word has a $\sigma_n$, delete it using the Markov Destabilization Rule. Then do this on the word again, noting that our highest subscript is now lower than it was before. With each repetition, $n$ decreases by at least one, so eventually the braid word will be the empty word. Because Lemma \ref{AA} terminates in a finite number of steps, we will always get to the empty word in a finite number of steps.
\end{proof}
\begin{thm}
\label{rules}
Let $\beta$ and $\beta'$ be positive braid words whose braid closures are knots. If $\beta$ can be made into $\beta'$ using the five rules, then $\hat\beta' \leq_g \hat\beta$.
\end{thm}
\begin{proof}
By Lemma \ref{two}, we know that we can turn $\beta'$ into the empty braid word using the five rules. Because the five rules do not change the number of components in the closure, that empty braid word corresponds to the unknot. Therefore, if we can go from $\beta$ to $\beta'$ using these rules, we can then combine that sequence with the sequence from $\beta'$ to the unknot to be able to go from $\beta$ to the unknot using only the five rules. Using Lemma \ref{one}, we have that this sequence is an unknotting sequence of $\hat\beta$. Therefore, $\hat\beta'$ is an intermediate knot on an unknotting sequence of $\hat\beta$, implying $\hat\beta' \leq_g \hat\beta$.
\end{proof}
Theorem \ref{rules} tells us that if we get from one positive braid word to another using these five rules, then we do not need to check how many crossing changes we have used, as the braid closure of the ending braid word is guaranteed to be Gordian adjacent to the braid closure of the starting braid word.

\section{Sufficient Conditions for Gordian Adjacency}

In this section, we present our main results on Gordian adjacency dealing with certain classes of positive braid knots. We do this constructively using the five rules introduced in the previous section. To shorten the sequence of knots generated by Lemma \ref{AA}, we formulate multi-step moves that are combinations of our rules. 
\begin{lem}
\label{iso1}
If $n$ is a positive integer and $\beta$ is a positive braid on $n+1$ strands with no $\sigma_n$ in its braid word, then $\beta\sigma_n\cdots\sigma_1\sigma_1\cdots\sigma_n=\sigma_n\cdots\sigma_1\sigma_1\cdots\sigma_n\beta$ as braids (not allowing conjugation or Markov destabilization).
\end{lem}
\begin{proof}
Let $1\le i\le n-1$. We will show that $\sigma_n\cdots\sigma_1\sigma_1\cdots\sigma_n$ commutes with $\sigma_i$. Since the braid word of $\beta$ consists of a series of $\sigma_i$ of this form, this implies the lemma.
\begin{align*}
&\sigma_n\cdots\sigma_1\sigma_1\cdots\sigma_n\sigma_i\\
&=\sigma_n\cdots\sigma_1\sigma_1\cdots\sigma_i\sigma_{i+1}\sigma_i\cdots\sigma_n\text{ through repeating the Distant Generators Rule}\\
&=\sigma_n\cdots\sigma_1\sigma_1\cdots\sigma_{i+1}\sigma_i\sigma_{i+1}\cdots\sigma_n\text{ through the Neighboring Generators Rule}\\
&=\sigma_n\cdots\sigma_{i+1}\sigma_i\sigma_{i+1}\cdots\sigma_1\sigma_1\cdots\sigma_i\sigma_{i+1}\cdots\sigma_n\text{ through repeating the Distant Generators Rule}\\
&=\sigma_n\cdots\sigma_i\sigma_{i+1}\sigma_i\cdots\sigma_1\sigma_1\cdots\sigma_n\text{ through the Neighboring Generators Rule}\\
&=\sigma_i\sigma_n\cdots\sigma_{i+1}\sigma_i\cdots\sigma_1\sigma_1\cdots\sigma_n\text{ through repeating the Distant Generators Rule}\\
&=\sigma_i\sigma_n\cdots\sigma_1\sigma_1\cdots\sigma_n.
\end{align*}
Since the Distant Generators Rule and the Neighboring Generators Rule are braid isotopies, we have that $\sigma_i$ commutes with $\sigma_n\cdots\sigma_1\sigma_1\cdots\sigma_n$ only through braid isotopies.

\end{proof}

We can also think about the proof of Lemma \ref{iso1} pictorially. The word $\sigma_n\cdots\sigma_1\sigma_1\cdots\sigma_n$ represents the furthest strand to the right on a braid on $n+1$ strands being wrapped once around the other $n$ strands. Any positive braid on the other $n$ strands has no interaction with the rightmost strand, and so can pass through the rightmost strand. Figure \ref{slide} shows this process for $n=5$.

\begin{figure}[H]
\begin{center}
\includegraphics[scale=.7]{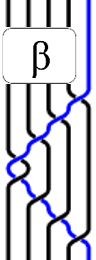}
\hspace{2em}
\begin{tikzpicture}
  \useasboundingbox (0,-2) rectangle (0,2);
  \node at (0,0) {$=$};
\end{tikzpicture}
\hspace{2em}
\includegraphics[scale=.7]{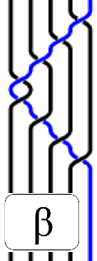}
\caption{Lemma \ref{iso1}}
\end{center}
\label{slide}
\end{figure}

The following standard definition will be of great use in proofs of our major results.

\begin{defn}\label{lemma1}
A full twist on $n$ strands, $\Delta_n^2$, is the braid $(\sigma_{n-1}\cdots\sigma_1)^n$. $\Delta_n^2$ is in the center of the braid group $B_n$, that is, $\Delta_n^2$ commutes with $\sigma_i$ for any $1\le i\le n-1$ \cite{StudyOfBraids}.
\end{defn}

\begin{lem}
\label{iso2}
If $n$ is a positive integer, then $\Delta_n^2 = \sigma_{n-1}\sigma_{n-2}\cdots\sigma_1\sigma_1\cdots\sigma_{n-2}\sigma_{n-1}(\Delta_{n-1}^2) \\= (\Delta_{n-1}^2)\sigma_{n-1}\sigma_{n-2}\cdots\sigma_1\sigma_1\cdots\sigma_{n-2}\sigma_{n-1}$ as braids (not allowing conjugation or Markov destabilization).
\end{lem}
\begin{proof}
By definition, $\Delta_n^2= (\sigma_{n-1}\cdots\sigma_1)(\sigma_{n-1}\cdots\sigma_1)^{n-1}$. Because $T(n,n-1)=T(n-1,n)$ \cite{KnotBook} and $(\sigma_1\cdots\sigma_{n-2})^n$ is a valid braid word for $T(n-1,n)$ as explained in Section 2, the closure of $(\sigma_{n-1}\cdots\sigma_1)^{n-1}$ is isotopic to that of $(\sigma_1\cdots\sigma_{n-2})^n$. Since $(\sigma_1\cdots\sigma_{n-2})\sigma_{n-1}(\sigma_1\cdots\sigma_{n-2})^{n-1}$ can be made into $(\sigma_1\cdots\sigma_{n-2})^n$ by the Markov Destabilization Rule, these two braids have isotopic closures, and thus the closure of $(\sigma_1\cdots\sigma_{n-2})\sigma_{n-1}(\sigma_1\cdots\sigma_{n-2})^{n-1}$ is also isotopic that of $(\sigma_{n-1}\cdots\sigma_1)^{n-1}$. Etnyre and Van Horn-Morris proved that any positive braids whose closure represent the same link are related by positive Markov moves, braid isotopy, and conjugation \cite{E-vHM}. Then as $(\sigma_{n-1}\cdots\sigma_1)^{n-1}$ and $(\sigma_1\cdots\sigma_{n-2})\sigma_{n-1}(\sigma_1\cdots\sigma_{n-2})^{n-1}$ are positive braids on the same number of strands with isotopic braid closures, these two braids must be related by the Conjugation Rule and isotopy. In fact, the repetitive nature of the braid words means that the Conjugation Rule is not necessary, as can be seen in Figure \ref{unfurl}. Since this relation requires only braid isotopies, it can be done on any section of the overall braid word. This implies $\Delta_n^2 = (\sigma_{n-1}\cdots\sigma_1)^n= (\sigma_{n-1}\cdots\sigma_1)(\sigma_{n-1}\cdots\sigma_1)^{n-1}=(\sigma_{n-1}\cdots\sigma_1)(\sigma_1\cdots\sigma_{n-2})\sigma_{n-1}(\sigma_1\cdots\sigma_{n-2})^{n-1}=\sigma_{n-1}\sigma_{n-2}\cdots\sigma_1\sigma_1\cdots\sigma_{n-2}\sigma_{n-1}(\Delta_{n-1}^2)$. Lemma \ref{iso1} tells us that these two sections commute.
\end{proof}

Lemma \ref{iso2} can be thought of pictorially as pulling the strand that starts furthest to the right as far up the braid as possible. Figure \ref{unfurl} shows this process for $n=4$.

\begin{figure}[H]
\begin{center}
\begin{tikzpicture}
  \braid[width=10pt, height=10pt, line width=2pt, style strands={4}{blue}] s_3^{-1} s_2^{-1} s_1^{-1} s_3^{-1} s_2^{-1} s_1^{-1} s_3^{-1} s_2^{-1} s_1^{-1} s_3^{-1} s_2^{-1} s_1^{-1};
\end{tikzpicture}
\hspace{2em}
\begin{tikzpicture}
  \useasboundingbox (0,0) rectangle (0,4);
  \node at (0,2) {$=$};
\end{tikzpicture}
\hspace{2em}
\begin{tikzpicture}
  \braid[width=10pt, height=10pt, line width=2pt, style strands={4}{blue}] s_3^{-1} s_2^{-1} s_1^{-1} s_1^{-1} s_2^{-1} s_3^{-1} s_2^{-1} s_1^{-1} s_2^{-1} s_1^{-1} s_2^{-1} s_1^{-1};
\end{tikzpicture}
\hspace{2em}
\begin{tikzpicture}
  \useasboundingbox (0,0) rectangle (0,4);
  \node at (0,2) {$=$};
\end{tikzpicture}
\hspace{2em}
\begin{tikzpicture}
  \braid[width=10pt, height=10pt, line width=2pt, style strands={4}{blue}] s_2^{-1} s_1^{-1} s_2^{-1} s_1^{-1} s_2^{-1} s_1^{-1} s_3^{-1} s_2^{-1} s_1^{-1} s_1^{-1} s_2^{-1} s_3^{-1};
\end{tikzpicture}
\end{center}
\caption{Lemma \ref{iso2}}
\label{unfurl}
\end{figure}

Bear in mind that Lemma \ref{iso2} can be repeated multiple times to obtain 
\begin{align*}
\Delta_n^2&=(\sigma_{n-1}\cdots\sigma_1\sigma_1\cdots\sigma_{n-1})(\Delta_{n-1}^2)\\&=(\sigma_{n-1}\cdots\sigma_1\sigma_1\cdots\sigma_{n-1})(\sigma_{n-2}\cdots\sigma_1\sigma_1\cdots\sigma_{n-2})(\Delta_{n-2}^2)\\&=(\sigma_{n-1}\cdots\sigma_1\sigma_1\cdots\sigma_{n-1})(\sigma_{n-2}\cdots\sigma_1\sigma_1\cdots\sigma_{n-2})\cdots(\sigma_2\sigma_1\sigma_1\sigma_2)(\sigma_1\sigma_1).
\end{align*} By Lemma \ref{iso1}, each of the groups in parentheses above can all commute with one another. Furthermore, these lemmas can be done on any subsequence of our braid word as they only use braid isotopy and follow the five rules on our overall braid word.

With these relations in mind, we now construct Gordian adjacencies between specific classes of torus knots. In each step of the following proofs, we do exactly one of the following:

\begin{itemize}
    \item rewrite the braid word in a different form for better clarity,

    \item commute sections of our braid word using some combination of Lemma \ref{iso1} and the Distant Generators Rule,

    \item use one of Lemma \ref{iso2}, Markov Destabilization Rule, or Neighboring Generators Rule on some section or sections of our braid word,

    \item make crossing changes via the Crossing Change Rule.
\end{itemize}

\begin{lem}
\label{iso3}
If $n$ and $k$ are positive integers with $n\ge 2$, then \\$(\sigma_{n-1}\cdots\sigma_1)^{nk+1}=(\sigma_{n-1}\cdots\sigma_1\sigma_1\cdots\sigma_{n-1})^k\sigma_{n-1}\cdots(\sigma_2\sigma_1\sigma_1\sigma_2)^k\sigma_2(\sigma_1\sigma_1)^k\sigma_1$ using braid isotopies and the Conjugation Rule.
\end{lem}
\begin{proof}
We proceed by induction on $n$. When $n=2$, $(\sigma_1)^{2k+1}=(\sigma_1\sigma_1)^k\sigma_1$. Now assume that this theorem holds true when $n=a$ for some $a\ge 2$. When $n=a+1$, we find
\begin{align*}
&(\sigma_a\cdots\sigma_1)^{(a+1)k+1}
\\&=(\Delta_{a+1}^2)^k\sigma_a\cdots\sigma_1 \\&=(\Delta_a^2\sigma_a\cdots\sigma_1\sigma_1\cdots\sigma_a)^k\sigma_a\cdots\sigma_1\text{ by Lemma \ref{iso2}} \\&=(\Delta_a^2)^k(\sigma_a\cdots\sigma_1\sigma_1\cdots\sigma_a)^k\sigma_a\cdots\sigma_1\text{ by Lemma \ref{iso1}} \\&=(\sigma_a\cdots\sigma_1\sigma_1\cdots\sigma_a)^k\sigma_a\cdots\sigma_1(\Delta_a^2)^k\text{ by the Conjugation Rule} \\&=(\sigma_a\cdots\sigma_1\sigma_1\cdots\sigma_a)^k\sigma_a(\Delta_a^2)^k\sigma_{a-1}\cdots\sigma_1\text{ as $\Delta_a^2$ commutes with $\sigma_i$ for all $i\le a-1$}
\\&=(\sigma_a\cdots\sigma_1\sigma_1\cdots\sigma_a)^k\sigma_a(\sigma_{a-1}\cdots\sigma_1)^{ak+1} \\&=(\sigma_a\cdots\sigma_1\sigma_1\cdots\sigma_a)^k\sigma_a(\sigma_{a-1}\cdots\sigma_1\sigma_1\cdots\sigma_{a-1})^k\sigma_{a-1}\cdots(\sigma_2\sigma_1\sigma_1\sigma_2)^k\sigma_2(\sigma_1\sigma_1)^k\sigma_1.
\end{align*}
Hence we have that the last step follows from our inductive hypothesis.
\end{proof}

Note that the conversion in the preceding lemma is bidirectional, as all braid isotopies and the Conjugation Rule are reversible. Later, we will use the reverse direction of Lemma \ref{iso3}. Since Lemmas \ref{iso1}, \ref{iso2}, and \ref{iso3} are combinations of our five rules, they follow Theorem \ref{rules}, implying each of these processes results in Gordian adjacencies.

We now focus on constructing Gordian adjacencies between torus knots whose indices differ by one. When going from $T(a,b)$ to $T(a-1,c)$ using the five rules, the first question that arises is how many crossing changes are required to eliminate all $\sigma_{a-1}$ from our braid word. The theorem below shows that this process can be done in $\left\lfloor\frac{b}{a}\right\rfloor$ steps.

\begin{prop}
\label{sign}
  For all positive integers $a$ and $b$ such that $a$ is coprime to $b$, there exists $\beta$ in $B_{a-1}$ such that $\hat\beta \leq_g T(a,b)$ and $d_g(T(a,b),\hat\beta) = \left\lfloor\frac{b}{a}\right\rfloor$.
\end{prop}
\begin{proof}
$T(a,b)$ can be represented by the braid word $(\Delta_a^2)^{\left\lfloor\frac{b}{a}\right\rfloor}(\sigma_{a-1}\cdots\sigma_1)^{(b\text{ (mod }a))}$. Taking the second section of this word, as $T(a,b\text{ (mod }a))=T(b\text{ (mod }a),a)$ and $b\text{ (mod }a) < a$, the braid word of $T(a,b\text{ (mod }a))$ can be changed into the braid word of $T(b\text{ (mod }a),a)$ using only braid isotopies and Markov destabilizations \cite{E-vHM}. Note that we can still use the Conjugation Rule because $\Delta_a^2$ commutes with all letters in our braid word. So there is a path using our rules from $T(a,b\text{ (mod }a))$ to $T(b\text{ (mod }a),a)$. On this path, we can pause before the Markov Destabilization Rule is used for the first time. At this point, the subword has only one copy of $\sigma_{a-1}$. Moreover, we have only used our rules on our overall braid word as we have only used braid isotopies.

By Lemma \ref{iso2},
\begin{align*}
(\Delta_a^2)^{\left\lfloor\frac{b}{a}\right\rfloor}&=(\sigma_{a-1}\cdots\sigma_1\sigma_1\cdots\sigma_{a-1}\Delta_{a-1}^2)^{\left\lfloor\frac{b}{a}\right\rfloor}\\&=(\sigma_{a-1}\cdots\sigma_1\sigma_1\cdots\sigma_{a-1})^{\left\lfloor\frac{b}{a}\right\rfloor}(\Delta_{a-1}^2)^{\left\lfloor\frac{b}{a}\right\rfloor}.
\end{align*}
We can then commute $(\sigma_{a-1}\cdots\sigma_1\sigma_1\cdots\sigma_{a-1})^{\left\lfloor\frac{b}{a}\right\rfloor}$ so that it is next to the final $\sigma_{a-1}$ made by the process above. We know this is possible because for every other $\sigma_i$ in that section of the braid word, $i<a-1$, so $\sigma_i$ commutes with $(\sigma_{a-1}\cdots\sigma_1\sigma_1\cdots\sigma_{a-1})$ by Lemma \ref{iso1}. This results in the section of a braid word $(\sigma_{a-1}\cdots\sigma_1\sigma_1\cdots\sigma_{a-1})^{\left\lfloor\frac{b}{a}\right\rfloor}\sigma_{a-1}$ which is the only section of the braid word that contains $\sigma_{a-1}$. We use the Crossing Change Rule to delete all consecutive pairs of $\sigma_{a-1}$. Since there are $\left\lfloor\frac{b}{a}\right\rfloor$ of these pairs, we need to use $\left\lfloor\frac{b}{a}\right\rfloor$ crossing changes. This leaves only one $\sigma_{a-1}$, which we can delete using the Markov Destabilization Rule. We are left with $\beta$. Gordian adjacency follows as we have only used the five rules.
\end{proof}

For the following proofs of our theorems and lemmas, we abbreviate ``crossing changes" as ``$CCs$" for brevity. 
\begin{thm}
\label{ci} 
If $n$ and $k$ are positive integers, then  $T(n,n^2k+1) \leq_g T(n+1,(n^2-1)k+1)$.
\end{thm}
\begin{proof}

We first remove all $\sigma_n$ from our braid word using the process described in Proposition \ref{sign}. 
\begin{align*}
&T(n+1,(n^2-1)k+1) \\
&= (\sigma_n\cdots\sigma_1)^{(n+1)(n-1)k+1} \\
&= (\Delta_{n+1}^2)^{(n-1)k}\sigma_n\cdots\sigma_1 \\
&= (\sigma_n\cdots\sigma_1\sigma_1\cdots\sigma_n\Delta_n^2)^{(n-1)k}\sigma_n\cdots\sigma_1\text{ by Lemma \ref{iso2}}\\
&= (\Delta_n^2)^{(n-1)k}(\sigma_n\cdots\sigma_1\sigma_1\cdots\sigma_n)^{(n-1)k}\sigma_n\cdots\sigma_1 \\
&\rightarrow(\Delta_n^2)^{(n-1)k}\sigma_n(\sigma_{n-1}\cdots\sigma_1\sigma_1\cdots\sigma_{n-1})^{(n-1)k}\sigma_{n-1}\cdots\sigma_1\text{(applying $(n-1)k$ $CCs$)}\\
&= (\Delta_n^2)^{(n-1)k}((\sigma_{n-1}\cdots\sigma_1\sigma_1\cdots\sigma_{n-1})^k)^{n-1}\sigma_{n-1}\cdots\sigma_1.
\end{align*}
Next, we remove $\sigma_{n-1}$ from all but one of the $n-1$ copies of $(\sigma_{n-1}\cdots\sigma_1\sigma_1\cdots\sigma_{n-1})^k$, then remove $\sigma_{n-2}$ from all of the copies from which $\sigma_{n-1}$ was removed except one, and so on:
\begin{align*}
&(\Delta_n^2)^{(n-1)k}((\sigma_{n-1}\cdots\sigma_1\sigma_1\cdots\sigma_{n-1})^k)^{n-1}\sigma_{n-1}\cdots\sigma_1\\&\rightarrow(\Delta_n^2)^{(n-1)k}(\sigma_{n-1}\cdots\sigma_1\sigma_1\cdots\sigma_{n-1})^k\sigma_{n-1}((\sigma_{n-2}\cdots\sigma_1\sigma_1\cdots\sigma_{n-2})^k)^{n-2}\sigma_{n-2}\cdots\sigma_1\text{ (applying $(n-2)k$ $CCs$)} \\&\hspace{7cm}\vdots \\&\rightarrow (\Delta_n^2)^{(n-1)k}(\sigma_{n-1}\cdots\sigma_1\sigma_1\cdots\sigma_{n-1})^k\sigma_{n-1}\cdots(\sigma_3\sigma_2\sigma_1\sigma_1\sigma_2\sigma_3)^{3k}\sigma_3\sigma_2\sigma_1\text{ (applying $3k$ $CCs$)} \\&\rightarrow (\Delta_n^2)^{(n-1)k}(\sigma_{n-1}\cdots\sigma_1\sigma_1\cdots\sigma_{n-1})^k\sigma_{n-1}\cdots(\sigma_3\sigma_2\sigma_1\sigma_1\sigma_2\sigma_3)^{k}\sigma_3(\sigma_2\sigma_1\sigma_1\sigma_2)^{2k}\sigma_2\sigma_1\text{ (applying $2k$ $CCs$)} \\&\rightarrow (\Delta_n^2)^{(n-1)k}(\sigma_{n-1}\cdots\sigma_1\sigma_1\cdots\sigma_{n-1})^k\sigma_{n-1}\cdots(\sigma_3\sigma_2\sigma_1\sigma_1\sigma_2\sigma_3)^{k}\sigma_3(\sigma_2\sigma_1\sigma_1\sigma_2)^k\sigma_2(\sigma_1\sigma_1)^k\sigma_1\text{ (applying $k$ $CCs$)}
\end{align*}

Lemma \ref{iso3} tells us that our braid word, excluding the initial $(\Delta_n^2)^{(n-1)k}$, can be converted to $(\sigma_{n-1}\cdots\sigma_1)^{nk+1}$ using braid isotopy and the Conjugation Rule. Bear in mind that we can still use the Conjugation Rule on this subword here because $\Delta_n^2$ commutes with everything. This implies
\begin{align*}
&(\Delta_n^2)^{(n-1)k}(\sigma_{n-1}\cdots\sigma_1\sigma_1\cdots\sigma_{n-1})^k\sigma_{n-1}\cdots(\sigma_3\sigma_2\sigma_1\sigma_1\sigma_2\sigma_3)^{k}\sigma_3(\sigma_2\sigma_1\sigma_1\sigma_2)^k\sigma_2(\sigma_1\sigma_1)^k\sigma_1 \\&=(\Delta_n^2)^{(n-1)k}(\sigma_{n-1}\cdots\sigma_1)^{nk+1}\text{ using Lemma \ref{iso3}}
\\&=(\sigma_{n-1}\cdots\sigma_1)^{n^2k+1}\\&=T(n,n^2k+1).
\end{align*}

Note that $u(T(n+1,(n^2-1)k+1)-u(T(n,n^2k+1)=\frac{(n)((n^2-1)k)}{2}-\frac{(n-1)(n^2k)}{2}=\frac{n^2k-nk}{2}=\frac{(n)(n-1)}{2}k$. The number of crossing changes we have made is $(n-1)k+(n-2)k+\cdots+2k+k=\frac{(n)(n-1)}{2}k$, which is exactly the difference in their unknotting numbers. Therefore, the torus knot $T(n,n^2k+1) \leq_g T(n+1,(n^2-1)k+1)$.
\end{proof}

\begin{thm}
\label{cin}
If $n$ and $k$ are positive integers, then  $T(n,n^2k+n+1) \leq_g T(n+1,(n^2-1)k+n)$.
\end{thm}
\begin{proof}
We first note that by isotopy, the subword $(\sigma_n\cdots\sigma_1)^n$ can be changed into $\sigma_n(\sigma_{n-1}\cdots\sigma_1)^{n+1}$ using identical logic as in the proof of Lemma \ref{iso2}.
\begin{align*}
&T(n+1,(n^2-1)k+n) \\&= (\sigma_n\cdots\sigma_1)^{(n+1)(n-1)k+n} \\&= ((\sigma_n\cdots\sigma_1)^{(n+1)})^{(n-1)k}(\sigma_n\cdots\sigma_1)^n \\&= (\Delta_{n+1}^2)^{(n-1)k}\sigma_n(\sigma_{n-1}\cdots\sigma_1)^{n+1}  \\&= (\Delta_{n+1}^2)^{(n-1)k}\sigma_n\sigma_{n-1}\cdots\sigma_1(\sigma_{n-1}\cdots\sigma_1)^n \\&= (\Delta_{n+1}^2)^{(n-1)k}\sigma_n\cdots\sigma_1\Delta_n^2 \\&= (\sigma_n\cdots\sigma_1\sigma_1\cdots\sigma_n\Delta_n^2)^{(n-1)k}\sigma_n\cdots\sigma_1\Delta_n^2 \\&= (\Delta_n^2)^{(n-1)k}(\sigma_n\cdots\sigma_1\sigma_1\cdots\sigma_n)^{(n-1)k}\sigma_n\cdots\sigma_1\Delta_n^2 \\&=(\Delta_n^2)^{(n-1)k+1}(\sigma_n\cdots\sigma_1\sigma_1\cdots\sigma_n)^{(n-1)k}\sigma_n\cdots\sigma_1.
\end{align*}
From here, we proceed in the same way as in the proof of Theorem \ref{ci}. Since we have one extra $\Delta_n^2$, we end with $T(n,n^2k+n+1)$ as the second number is increased by $n$ from the previous case.
\end{proof}

When proving adjacencies between torus knots of the form $T(n,a)$ and $T(n+1,b)$, Theorems \ref{ci} and \ref{cin} cover two cases: when $b$ is congruent to either 1 or $n$ modulo $n^2-1$ respectively. In general, the number of cases is equal to the amount of numbers between 1 and $n^2-1$ inclusive that are coprime with $n+1$, which equals $(n-1)(\phi(n+1))$, where $\phi$ is the Euler Totient function. For $n=2$, this value is 2, so all cases are considered. In fact, the combinations of Theorems \ref{ci} and \ref{cin} give an alternate proof of one direction of Feller's Theorem 3 (Theorem \ref{fel3}), as it comes as a corollary that $T(2,a) \leq_g (3,b)$ if $a\le\frac{4b+1}{3}$. Next, we consider the case when $n=3$, implying that $b$ must be congruent to either 1, 3, 5, or 7 modulo $8=3^2-1$. These cases are dealt with in the four lemmas below, the first two of which follow from Theorems 4.1 and 4.2 respectively.

\begin{lem}
\label{a}
If $k$ is a positive integer, then  $T(3,9k+1) \leq_g T(4,8k+1)$.
\end{lem}

\begin{lem}
\label{b}
If $k$ is a positive integer, then  $T(3,9k+4) \leq_g T(4,8k+3)$.
\end{lem}

\begin{lem}
\label{c}
If $k$ is a positive integer, then  $T(3,9k+5) \leq_g T(4,8k+5)$.
\end{lem}
\begin{proof}
First, we eliminate all $\sigma_3$ from our braid word by a similar process to Proposition \ref{sign}. We have that
\begin{align*}
&T(4,8k+5)
\\&=(\Delta_4^2)^{2k+1}\sigma_3\sigma_2\sigma_1\\&=(\Delta_3^2\sigma_3\sigma_2\sigma_1\sigma_1\sigma_2\sigma_3)^{2k+1}\sigma_3\sigma_2\sigma_1\\&=(\Delta_3^2)^{2k+1}(\sigma_3\sigma_2\sigma_1\sigma_1\sigma_2\sigma_3)^{2k+1}\sigma_3\sigma_2\sigma_1\\&\rightarrow(\Delta_3^2)^{2k+1}\sigma_3(\sigma_2\sigma_1\sigma_1\sigma_2)^{2k+1}\sigma_2\sigma_1\text{ (applying $2k+1$ $CCs$)}\\&=(\Delta_3^2)^{2k+1}(\sigma_2\sigma_1\sigma_1\sigma_2)^{2k+1}\sigma_2\sigma_1
\end{align*}
Then, we turn our braid on three strands into the required torus knot using similar ideas to Theorem \ref{ci}, with some changes to account for the extra copy of $\sigma_2\sigma_1\sigma_1\sigma_2$.
\begin{align*}
&=(\Delta_3^2)^{2k+1}(\sigma_2\sigma_1\sigma_1\sigma_2)^k\sigma_2\sigma_1\sigma_1\sigma_2(\sigma_2\sigma_1\sigma_1\sigma_2)^k\sigma_2\sigma_1
\\&\rightarrow (\Delta_3^2)^{2k+1}(\sigma_2\sigma_1\sigma_1\sigma_2)^k\sigma_2\sigma_1\sigma_1(\sigma_1\sigma_1)^k\sigma_1\text{ (applying $k+1$ $CCs$)}
\\&=(\Delta_3^2)^{2k+1}(\sigma_2\sigma_1\sigma_1\sigma_2)^k(\sigma_1\sigma_1)^k\sigma_1\sigma_2\sigma_1\sigma_1\text{ commuting $\sigma_1$'s around the end.}\\&= (\Delta_3^2)^{2k+1}(\sigma_2\sigma_1\sigma_1\sigma_2\sigma_1\sigma_1)^k\sigma_1\sigma_2\sigma_1\sigma_1 \\&= (\Delta_3^2)^{2k+1}(\sigma_2\sigma_1\sigma_2\sigma_1\sigma_2\sigma_1)^k\sigma_2\sigma_1\sigma_2\sigma_1
\\&= (\sigma_2\sigma_1)^{3(2k+1)+3(k)+2} \\&= (\sigma_2\sigma_1)^{9k+5} \\&= T(3,9k+5).
\end{align*}
Note that $u(T(4,8k+5))-u(T(3,9k+5)) = (12k+6)-(9k+4) = 3k+2$, which is exactly the number of crossing changes we have used. Therefore, the torus knot $T(3,9k+1) \leq_g T(4,8k+1)$.
\end{proof}

\begin{lem}
\label{d}
If $k$ is a positive integer, then  $T(3,9k+8) \leq_g T(4,8k+7)$.
\end{lem}
\begin{proof}
As in the previous lemma, we start by removing all $\sigma_3$ from our braid word.
\begin{align*}
&T(4,8k+7)\\&=((\sigma_3\sigma_2\sigma_1)^4)^{2k+1}\sigma_3\sigma_2\sigma_1\sigma_3\sigma_2\sigma_1\sigma_3\sigma_2\sigma_1\\&=(\Delta_4^2)^{2k+1}\sigma_1\sigma_2\sigma_3\sigma_2\sigma_1\sigma_2\sigma_1\sigma_2\sigma_1\\&=(\Delta_3^2\sigma_3\sigma_2\sigma_1\sigma_1\sigma_2\sigma_3)^{2k+1}\sigma_1\sigma_2\sigma_3\Delta_3^2\\&=(\Delta_3^2)^{2k+2}\sigma_1\sigma_2(\sigma_3\sigma_2\sigma_1\sigma_1\sigma_2\sigma_3)^{2k+1}\sigma_3\\&\rightarrow(\Delta_3^2)^{2k+2}\sigma_1\sigma_2\sigma_3(\sigma_2\sigma_1\sigma_1\sigma_2)^{2k+1}\text{ (applying $2k+1$ $CCs$)}\\&=(\Delta_3^2)^{2k+2}\sigma_1\sigma_2(\sigma_2\sigma_1\sigma_1\sigma_2)^{2k+1}    
\end{align*}
Then, we once more convert our three-strand braid to the required form, using a few more steps to deal with the technicalities from the added generators.
\begin{align*}
&=(\Delta_3^2)^{2k+2}\sigma_1\sigma_2(\sigma_2\sigma_1\sigma_1\sigma_2)^k\sigma_2\sigma_1\sigma_1\sigma_2(\sigma_2\sigma_1\sigma_1\sigma_2)^k
\\&\rightarrow (\Delta_3^2)^{2k+2}\sigma_1(\sigma_1\sigma_1)^k\sigma_1\sigma_1\sigma_2(\sigma_2\sigma_1\sigma_1\sigma_2)^k\text{ (applying $k+1$ $CCs$)}  \\&=(\Delta_3^2)^{2k+2}(\sigma_1\sigma_1)^k\sigma_1\sigma_1\sigma_1\sigma_2(\sigma_2\sigma_1\sigma_1\sigma_2)^k \\&=(\Delta_3^2)^{2k+2}\sigma_1\sigma_1\sigma_1\sigma_2(\sigma_2\sigma_1\sigma_1\sigma_2)^k(\sigma_1\sigma_1)^k\text{ commuting $\sigma_1$'s around the end.}\\&= (\Delta_3^2)^{2k+2}\sigma_1\sigma_1\sigma_1\sigma_2(\sigma_2\sigma_1\sigma_1\sigma_2\sigma_1\sigma_1)^k \\&= (\Delta_3^2)^{2k+2}\sigma_1\sigma_1\sigma_1\sigma_2(\sigma_2\sigma_1\sigma_2\sigma_1\sigma_2\sigma_1)^k \\&= ((\sigma_2\sigma_1)^3)^{2k+2}\sigma_1\sigma_2\sigma_1\sigma_1((\sigma_2\sigma_1)^3)^k\text{ commuting $\sigma_1$'s around the end.} \\&= ((\sigma_2\sigma_1)^3)^{2k+2}\sigma_2\sigma_1\sigma_2\sigma_1((\sigma_2\sigma_1)^3)^k \\&= (\sigma_2\sigma_1)^{3(2k+2)+2+3(k)} \\&= (\sigma_2\sigma_1)^{9k+8} \\&= T(3,9k+8).
\end{align*}
Note that $u(T(4,8k+7))-u(T(3,9k+8)) = (12k+9)-(9k+7) = 3k+2$, which is exactly the number of crossing changes we have used. Therefore, the torus knot $T(3,9k+8) \leq_g T(4,8k+7)$.
\end{proof}

\begin{thm}
\label{thrfou}
Let $a$ and $b$ be positive integers such that $a$ is coprime to 3 and $b$ is coprime to 4. If $a\le\frac{9b+5}{8}$, then  $T(3,a) \leq_g T(4,b)$.
\end{thm}
\begin{proof}
This proof utilizes the four lemmas above, which cover the four possible cases for a torus knot of index four. We now go through each lemma and replace the number of twists in our 4-strand torus knot with $b$. For Lemma \ref{a}, $b=8k+1$, so $$T\left(3,9k+1\right)=T\left(3,9\left(\frac{b-1}{8}\right)+1\right)=T\left(3,\frac{9b}{8}-\frac{1}{8}\right)=T\left(3,\left\lfloor\frac{9b+5}{8}\right\rfloor\right)\le_g T(4,b).$$ For Lemma \ref{b}, $b=8k+3$, so $$T(3,9k+4)=T\left(3,9\left(\frac{b-3}{8}\right)+4\right)=T\left(3,\frac{9b}{8}+\frac{5}{8}\right)=T\left(3,\left\lfloor\frac{9b+5}{8}\right\rfloor\right)\le_g T(4,b).$$ For Lemma \ref{c}, $b=8k+5$, so $$T(3,9k+5)=T\left(3,9\left(\frac{b-5}{8}\right)+5\right)=T\left(3,\frac{9b}{8}-\frac{5}{8}\right)=T\left(3,\left\lfloor\frac{9b+5}{8}\right\rfloor-1\right)\le_g T(4,b).$$ This seems to break the trend, but note that $\left\lfloor\frac{9b+5}{8}\right\rfloor=\frac{9b+3}{8}$ in this case, which is divisible by 3 and is therefore not a torus knot. This implies we need to show the index one lower to prove the desired result. For Lemma \ref{d}, $b=8k+7$, so $$T(3,9k+8)=T\left(3,9\left(\frac{b-7}{8}\right)+8\right)=T\left(3,\frac{9b}{8}+\frac{1}{8}\right)=T\left(3,\left\lfloor\frac{9b+5}{8}\right\rfloor\right)\le_g T(4,b).$$ These results show that the torus knot $T\left(3,\left\lfloor\frac{9b+5}{8}\right\rfloor\right) \leq_g T(4,b)$. By Feller's Theorem 2 (Theorem \ref{fel2}), this means that the torus knot $T(3,a) \leq_g T(4,b)$ if $a\le\frac{9b+5}{8}$.
\end{proof}

We can also use a similar process to find Gordian adjacencies between torus knots whose indices differ by more than one. Below, we explore Gordian adjacencies between torus knots of index two and index four. Note that for $T(2,b)$ to be a torus knot, $b$ must be congruent to either 1 or 3 modulo 4. These provide the two cases we split our argument into. In each case, we first remove all $\sigma_3$ using a similar process to Proposition \ref{sign}, then after some technical manipulations remove all $\sigma_2$ in two consecutive stages.
\begin{lem}
\label{2a}
If $k$ is a positive integer, then  $T(2,6k+3) \leq_g T(4,4k+1)$.
\end{lem}
\begin{proof}
By definition, we have
\begin{align*}
&T(4,4k+1) \\&=(\sigma_3\sigma_2\sigma_1)^{4k+1} \\&=(\Delta_4^2)^k\sigma_3\sigma_2\sigma_1 \\&=(\Delta_3^2\sigma_3\sigma_2\sigma_1\sigma_1\sigma_2\sigma_3)^k\sigma_3\sigma_2\sigma_1 \\&=(\Delta_3^2)^k(\sigma_3\sigma_2\sigma_1\sigma_1\sigma_2\sigma_3)^k\sigma_3\sigma_2\sigma_1 \\&\rightarrow (\Delta_3^2)^k\sigma_3(\sigma_2\sigma_1\sigma_1\sigma_2)^k\sigma_2\sigma_1\text{ (applying $k$ $CCs$)} \\&=(\Delta_3^2)^k(\sigma_2\sigma_1\sigma_1\sigma_2)^k\sigma_2\sigma_1 \\&=(\sigma_2\sigma_1\sigma_2\sigma_1\sigma_2\sigma_1)^k(\sigma_2\sigma_1\sigma_1\sigma_2)^k\sigma_2\sigma_1 \\&=(\sigma_2\sigma_1\sigma_1\sigma_2\sigma_1\sigma_1)^k(\sigma_2\sigma_1\sigma_1\sigma_2)^k\sigma_2\sigma_1 \\&=(\sigma_2\sigma_1\sigma_1\sigma_2\sigma_1\sigma_1\sigma_2\sigma_1\sigma_1\sigma_2)^k\sigma_2\sigma_1 \\&\rightarrow \sigma_2(\sigma_1\sigma_1\sigma_2\sigma_1\sigma_1\sigma_2\sigma_1\sigma_1)^k\sigma_2\sigma_2\sigma_1\text{ (applying $k-1$ $CCs$)}
\\&=(\sigma_1\sigma_1\sigma_2\sigma_1\sigma_1\sigma_2\sigma_1\sigma_1)^k\sigma_1\sigma_2\sigma_1\sigma_1 \\&=(\sigma_1\sigma_1)^k(\sigma_2\sigma_1\sigma_1\sigma_2)^k(\sigma_1\sigma_1)^k\sigma_1\sigma_2\sigma_1\sigma_1 \\&=(\sigma_1\sigma_1)^{2k+1}(\sigma_2\sigma_1\sigma_1\sigma_2)^k\sigma_2\sigma_1 \\&\rightarrow (\sigma_1\sigma_1)^{2k+1}\sigma_2(\sigma_1\sigma_1)^k\sigma_1\text{ ($k$ $CCs$)}
\\&=(\sigma_1)^{2(2k+1)+2(k)+1} \\&=(\sigma_1)^{6k+3} \\&=T(2,6k+3).
\end{align*}
Note that $u(T(4,4k+1))-u(T(2,6k+3)) = 6k-(3k+1) = 3k-1$, which is exactly how many crossing changes we have used. Therefore, the torus knot $T(2,6k+3) \leq_g T(4,4k+1)$.
\end{proof}
\begin{lem}
\label{2b}
If $k$ is a positive integer, then  $T(2,6k+5) \leq_g T(4,4k+3)$.
\end{lem}
\begin{proof}
By definition, we have 
\begin{align*}
&T(4,4k+3) \\&=(\sigma_3\sigma_2\sigma_1)^{4k+3} \\&=(\Delta_4^2)^k\sigma_3\sigma_2\sigma_1\sigma_3\sigma_2\sigma_1\sigma_3\sigma_2\sigma_1 \\&=(\Delta_4^2)^k\sigma_3\sigma_2\sigma_1\sigma_2\sigma_3\sigma_2\sigma_1\sigma_2\sigma_1 \\&=(\Delta_3^2\sigma_3\sigma_2\sigma_1\sigma_1\sigma_2\sigma_3)^k\sigma_1\sigma_2\sigma_3\sigma_2\sigma_1\sigma_2\sigma_1\sigma_2\sigma_1  
\\&=(\Delta_3^2)^{k+1}(\sigma_3\sigma_2\sigma_1\sigma_1\sigma_2\sigma_3)^k\sigma_1\sigma_2\sigma_3 \\&=(\Delta_3^2)^{k+1}\sigma_1\sigma_2(\sigma_3\sigma_2\sigma_1\sigma_1\sigma_2\sigma_3)^k\sigma_3 \\&\rightarrow (\Delta_3^2)^{k+1}\sigma_1\sigma_2\sigma_3(\sigma_2\sigma_1\sigma_1\sigma_2)^k\text{ (applying $k$ $CCs$)} \\&=(\Delta_3^2)^{k+1}\sigma_1\sigma_2(\sigma_2\sigma_1\sigma_1\sigma_2)^k \\&=(\sigma_2\sigma_1\sigma_2\sigma_1\sigma_2\sigma_1)^{k+1}\sigma_1\sigma_2(\sigma_2\sigma_1\sigma_1\sigma_2)^k \\&=(\sigma_2\sigma_1\sigma_1\sigma_2\sigma_1\sigma_1)^{k+1}\sigma_1\sigma_2(\sigma_2\sigma_1\sigma_1\sigma_2)^k \\&=(\sigma_2\sigma_1\sigma_1\sigma_2\sigma_1\sigma_1)^{k}\sigma_2\sigma_1\sigma_1\sigma_2\sigma_1\sigma_1\sigma_1\sigma_2(\sigma_2\sigma_1\sigma_1\sigma_2)^k \\&=(\sigma_2\sigma_1\sigma_1\sigma_2\sigma_1\sigma_1)^{k}\sigma_1\sigma_1\sigma_1\sigma_2\sigma_1\sigma_1\sigma_2\sigma_2(\sigma_2\sigma_1\sigma_1\sigma_2)^k \\&\rightarrow(\sigma_2\sigma_1\sigma_1\sigma_2\sigma_1\sigma_1)^{k}\sigma_1\sigma_1\sigma_1\sigma_2\sigma_1\sigma_1(\sigma_2\sigma_1\sigma_1\sigma_2)^k\text{ (applying one $CC$)} \\&=(\sigma_2\sigma_1\sigma_1\sigma_2\sigma_1\sigma_1\sigma_2\sigma_1\sigma_1\sigma_2)^k\sigma_2(\sigma_1)^5 \\&\rightarrow \sigma_2(\sigma_1\sigma_1\sigma_2\sigma_1\sigma_1\sigma_2\sigma_1\sigma_1)^k(\sigma_1)^5\text{ (applying $k$ $CCs$)} \\&=\sigma_2(\sigma_2\sigma_1\sigma_1\sigma_2)^k(\sigma_1\sigma_1\sigma_1\sigma_1)^k(\sigma_1)^5 \\&\rightarrow (\sigma_1\sigma_1)^k\sigma_2(\sigma_1\sigma_1\sigma_1\sigma_1)^k(\sigma_1)^5\text{ (applying $k$ $CCs$)} \\&=(\sigma_1\sigma_1)^k(\sigma_1\sigma_1\sigma_1\sigma_1)^k(\sigma_1)^5 \\&=(\sigma_1)^{2(k)+4(k)+5} \\&=(\sigma_1)^{6k+5} \\&=T(2,6k+5).
\end{align*}
Note that $u(T(4,4k+3))-u(T(2,6k+5)) = 6k+3-(3k+2) = 3k+1$, which is exactly how many crossing changes we have used. Therefore, the torus knot $T(2,6k+5) \leq_g T(4,4k+3)$.
\end{proof}
\begin{thm}
\label{twofou}
Let $a$ and $b$ be odd positive integers. If $a\le\frac{3b+3}{2}$, then  $T(2,a) \leq_g T(4,b)$.
\end{thm}
\begin{proof}
This proof follows the same structure as the proof of Theorem \ref{thrfou}. For Lemma \ref{2a}, $b=4k+1$, so $$T(3,6k+3)=T\left(3,6\left(\frac{b-1}{4}\right)+3\right)=T\left(3,\frac{3b+3}{2}\right)\le_g T(4,b).$$ For Lemma \ref{2b}, $b=4k+3$, so $$T(3,6k+5)=T\left(3,6\left(\frac{b-3}{4}\right)+5\right)=T\left(3,\frac{3b+1}{2}\right)\le_g T(4,b).$$ Note that in the second case, $\frac{3b+3}{2}$ is even, so as before we only need to show the number one lower. By Feller's Theorem 2 (Theorem \ref{fel2}), we have that the torus knot $T(2,a) \leq_g T(4,b)$ if $a\le\frac{3b+3}{2}$.
\end{proof}
While we conjecture that the converse of Theorem \ref{thrfou} is true, we know that the converse of Theorem \ref{twofou} is not: we have found that $T(2,13) \leq_g T(4,7)$.

So far, all of the Gordian adjacencies presented have been between torus knots. We next consider which positive braid $w$ we can concatenate another positive braid $\beta'$ in $B_n$ with. In doing so, we want the concatenation to result in a longer positive braid $\beta$ in $B_n$ whose closure is Gordian adjacent to the closure of $\beta'$. Note that $w$ must always have even length for both $\hat\beta$ and $\hat\beta'$ to be knots, as the unknotting formula $\frac{\ell-n+1}{2}$ must result in an integer. If $w$ satisfies the following properties, Gordian adjacency can be guaranteed.
\begin{thm}\label{4.5}
\label{w}
Let $\beta$ in $B_n$ be a positive braid where $\beta=\beta'w$ and $\hat\beta$ and $\hat\beta'$ are knots. If $\hat w$ is a link with $n$ components, then $\hat\beta' \leq_g \hat\beta$.
\end{thm}
\begin{proof}
By Lemma \ref{AA}, we can use the five rules to make $\beta'w$ into a new braid word $\beta'w'$ where $w'$ has no more than one $\sigma_{n-1}$. As none of the five rules change the number of components in the braid closure of the subword we perform the rules on, $\hat w'$ must have $n$ components. If $w'$ had only one $\sigma_{n-1}$, we could delete the $\sigma_{n-1}$ using the Markov Destabilization Rule and express $w'$ on $n-1$ strands; since $\hat w'$ has $n$ components, this cannot be done. Therefore, $w'$ has no $\sigma_{n-1}$. Using Lemma \ref{AA} on $\beta'w'$ again leads to a $w''$ with no $\sigma_{n-2}$. By repeating this process $n$ times, we can turn $\beta'w$ into $\beta'$ using the five rules. Since there is a sequence using the five rules that goes from $\beta$ to $\beta'$, $\hat\beta' \leq_g \hat\beta$.
\end{proof}
Theorem \ref{4.5} tells us that we can add a braid word whose closure is an $n$-component link (for instance, $\Delta_n^2$) to any positive braid $\beta$ in $B_n$ to get a new knot $K$ such that that $\hat\beta \leq_g K$. In the other direction, deleting any subword whose closure is an $n$-component link preserves Gordian adjacency to the original braid word. This theorem also provides an alternate proof for a weakened version of Feller's Theorem 2 (Theroem \ref{fel2}), as it shows $T(a,c) \leq_g T(a,d)$ if $c=d-ka$ for some positive integer $k$.

\section{Positive Paths}
The concept of Gordian adjacency is closely tied to the concept of unknotting sequences, as a knot $K_1 \leq_g K_2$ if and only if $K_1$ is contained in an unknotting sequence of $K_2$ \cite{Feller}. In finding Gordian adjacencies between positive braid knots, we are really describing unknotting sequences that contain positive braid knots. The five rules give us further tools to explore these unknotting sequences.
\begin{defn}
A positive path from a positive braid knot $\hat\beta_2$ to a positive braid knot $\hat\beta_1$ is a subsequence of an unknotting sequence of $\hat\beta_2$, $\hat\beta_2 \to \hat\alpha_1 \to \ldots \to \hat\alpha_n \to \hat\beta_1$, such that each $\hat\alpha_i$ is a positive braid knot.
\end{defn}
Note that every Gordian adjacency we show in this paper is shown through a positive path.
\begin{thm}
\label{posunk}
Every positive braid knot has an unknotting sequence that is a positive path.
\end{thm}
\begin{proof}
The combination of Lemmas \ref{one} and \ref{two} tells us that every positive braid word that represents a knot can be unknotted using the five rules, and that this procedure forms an unknotting sequence. As the application of any one of our rules to any positive braid results in a positive braid, this unknotting sequence contains only positive braid knots.
\end{proof}

The next question we consider is how many unknotting sequences of positive braid knots are positive paths. Since every knot with an unknotting number greater than one has infinetly many unknotting sequences \cite{KnotBook}, we naturally ask ourselves whether there are an infinite number of unknotting sequences of a positive braid knot $\hat\beta$ that are positive paths. Our next theorem implies that any positive braid knot has only finitely many positive paths.

\begin{lem}
\label{gen1}
If $\beta$ is a positive braid in $B_n$ whose closure is a knot, then each generator $\sigma_i$ for $1 \leq i \leq n-1$ must appear at least once in $\beta$.
\end{lem}
\begin{proof}
Assume $\beta$ is a positive braid in $B_n$ that has no $\sigma_i$ for some $1 \leq i \leq n-1$. Note that $\beta=\beta_1\beta_2$ for positive braids $\beta_1$ and $\beta_2$, where $j<i$ for all $\sigma_j$ in $\beta_1$ and $j>i$ for all $\sigma_j$ in $\beta_2$. We know this can be done via the Distant Generators Rule because any generator that ends up in $\beta_2$ has subscript at least two greater than the subscript of any generator that ends up in $\beta_1$. As shown by the figure below, we see that $\hat\beta$ must be a link, as it has at least two components $\hat\beta_1$ and $\hat\beta_2$. By the contrapositive, if $\hat\beta$ is a knot, then there are no $\sigma_i$ that $\beta$ does not contain.
\end{proof}

\begin{figure}[H]
\begin{center}
	\includegraphics[scale=0.20]{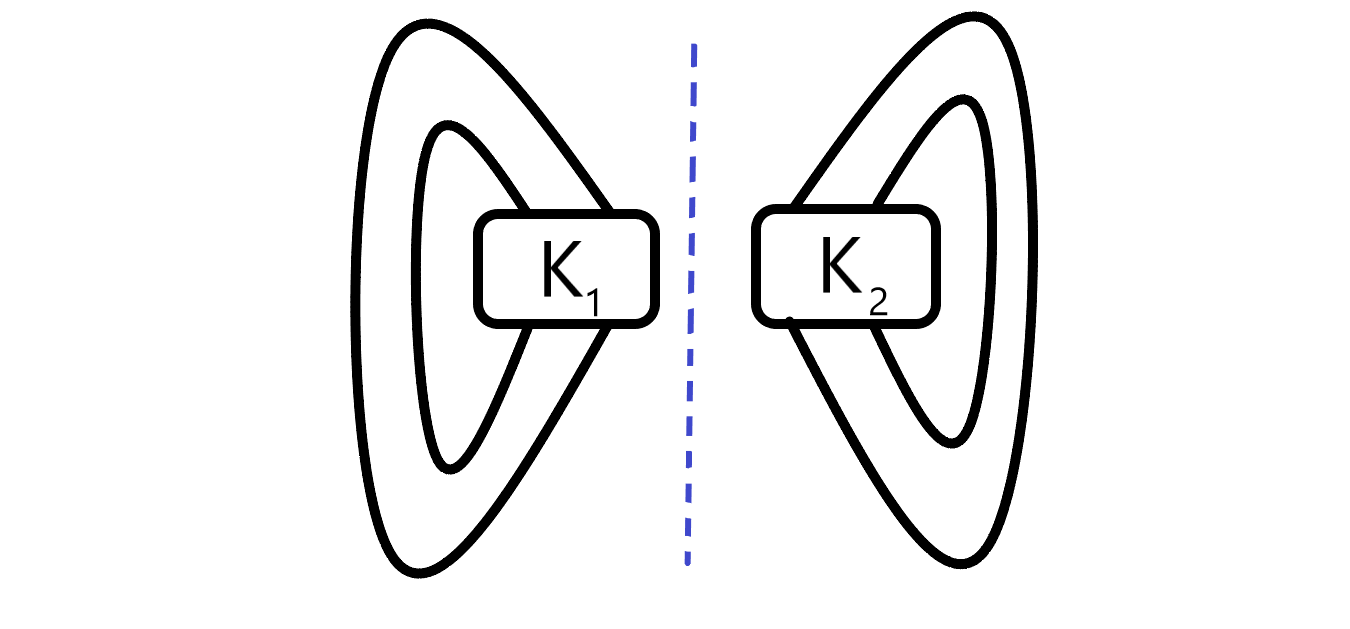}
\end{center}
\caption{A Braid with No $\sigma_i$}
\end{figure}

\begin{lem}
\label{gen2}
Let $\hat\beta$ be a positive braid knot such that $\beta$ is in $B_n$. If there exists $1 \le i \le n-1$ such that $\sigma_i$ appears exactly once in $\beta$, then $\hat\beta$ can be expressed as the closure of a positive braid $\beta'$ in $B_{n-1}$.
\end{lem}
\begin{proof}
We first note that $\beta=\beta_1\sigma_i\beta_2$ for positive braids $\beta_1$ and $\beta_2$, where $j<i$ for all generators $\sigma_j$ in $\beta_1$ and $j>i$ for all $\sigma_j$ in $\beta_2$. Similarly to the previous lemma, each generator can be moved to the correct side of the braid word using the Distant Generators Rule and the Conjugation Rule, as the generator gets moved around the end of the braid word while avoiding $\sigma_i$. This works because any generator that ends up in $\beta_2$ has subscript at least two greater than the subscript of any generator that ends up in $\beta_1$.

After expressing $\beta$ in this form, in Figure \ref{10ref} below, we see that we can isotopy the closure of $\beta_1\sigma_i\beta_2$ into the closure of $\beta_1\beta_3$, where $\beta_3$ is the same braid word as $\beta_2$, but with each subscript decreased by one. Since $\beta_1$ and $\beta_3$ are in $B_{n-1}$, this completes the proof.
\end{proof}

\begin{figure}[H]
\begin{center}
\begin{tikzpicture}
  \node[](l1){\includegraphics[scale=.2]{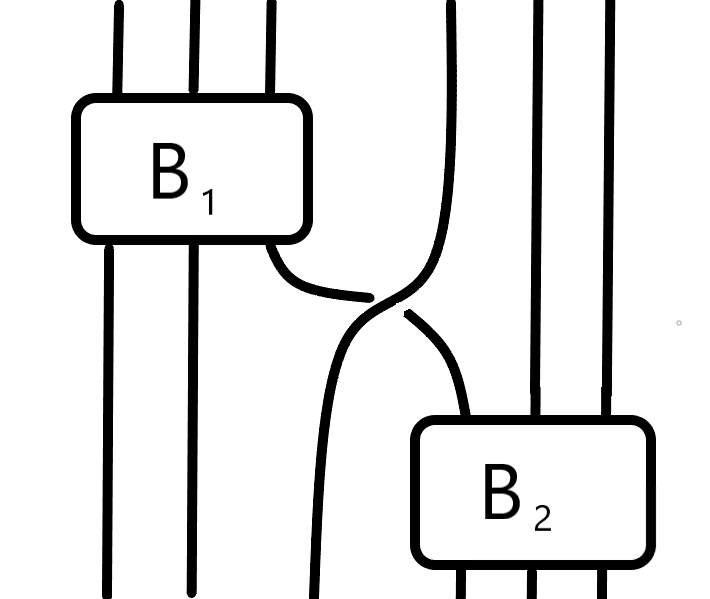}};
  \end{tikzpicture}
\begin{tikzpicture}
  \node[right=5cm](l2){\includegraphics[scale=.13]{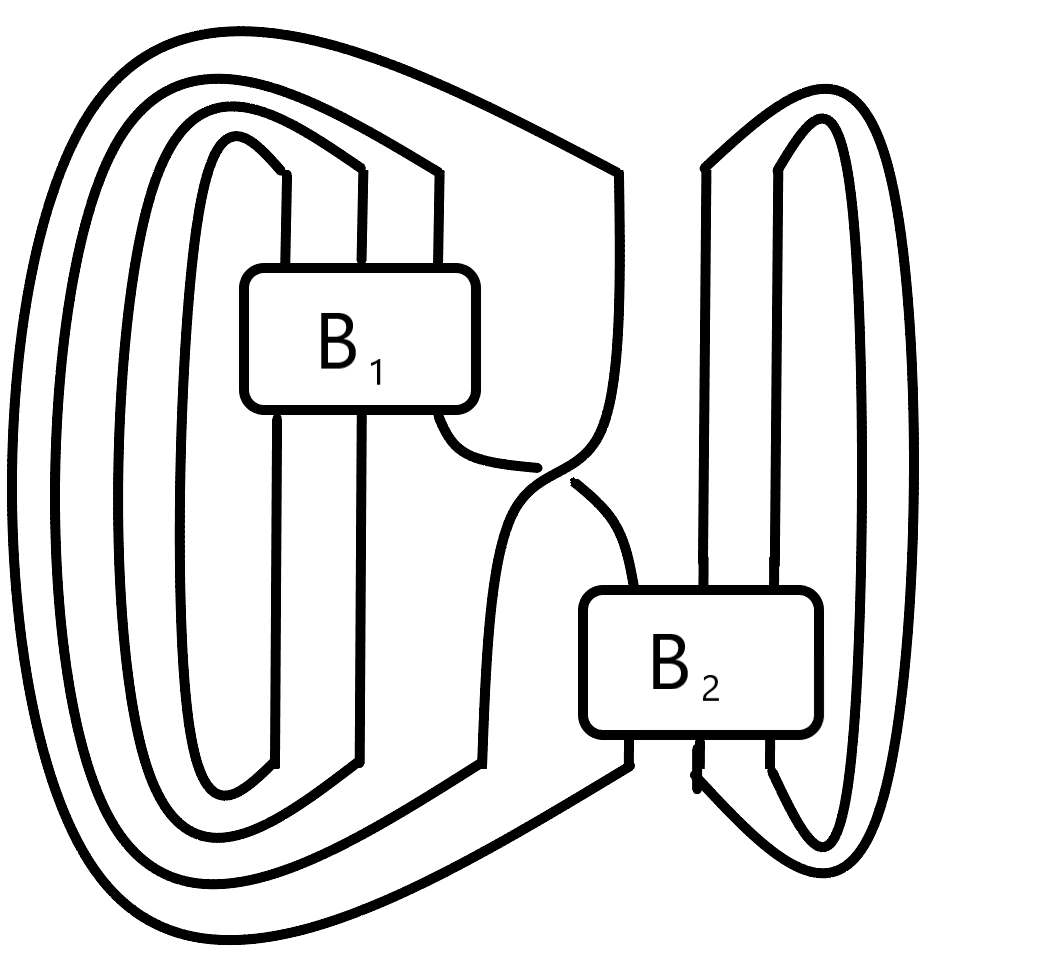}};
  \end{tikzpicture}
\begin{tikzpicture}
  \node[right=5cm](r1){\includegraphics[scale=.13]{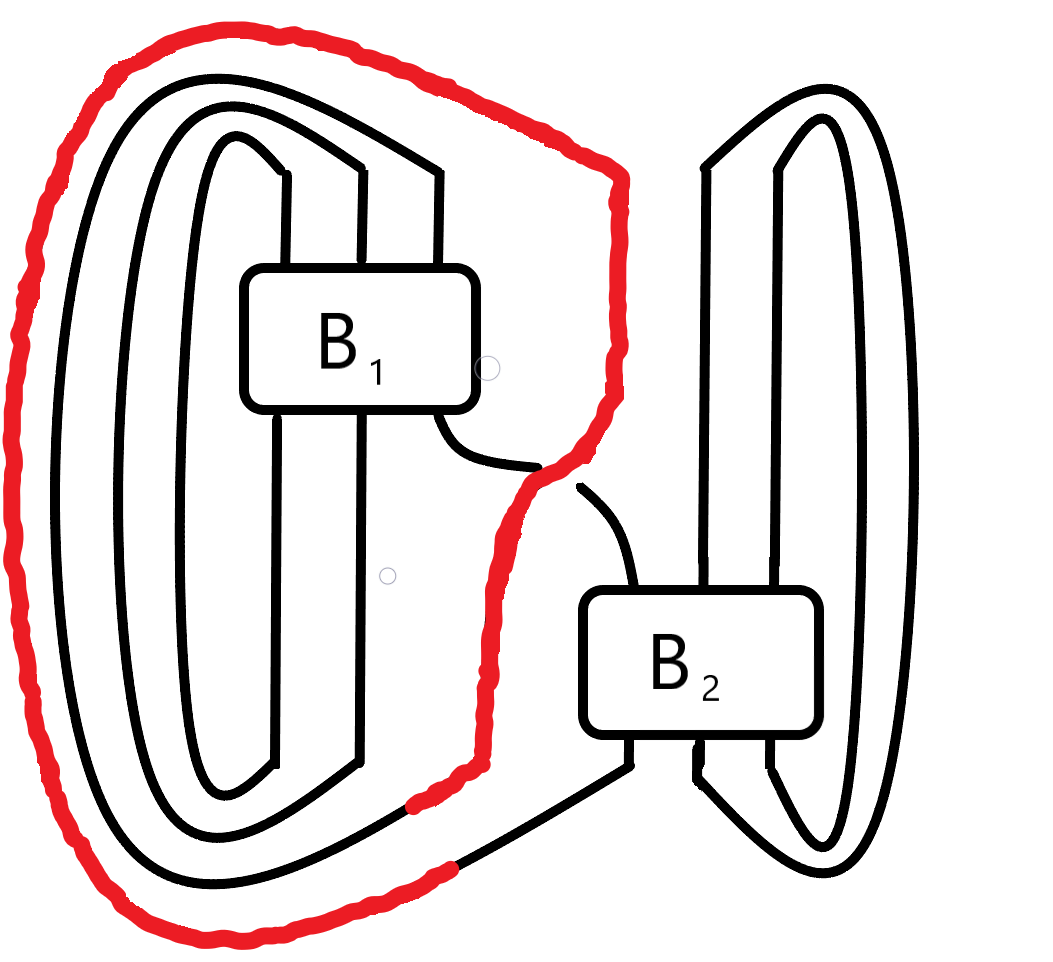}};
  \end{tikzpicture}
\begin{tikzpicture}
  \node[right=5cm](r2){\includegraphics[scale=.13]{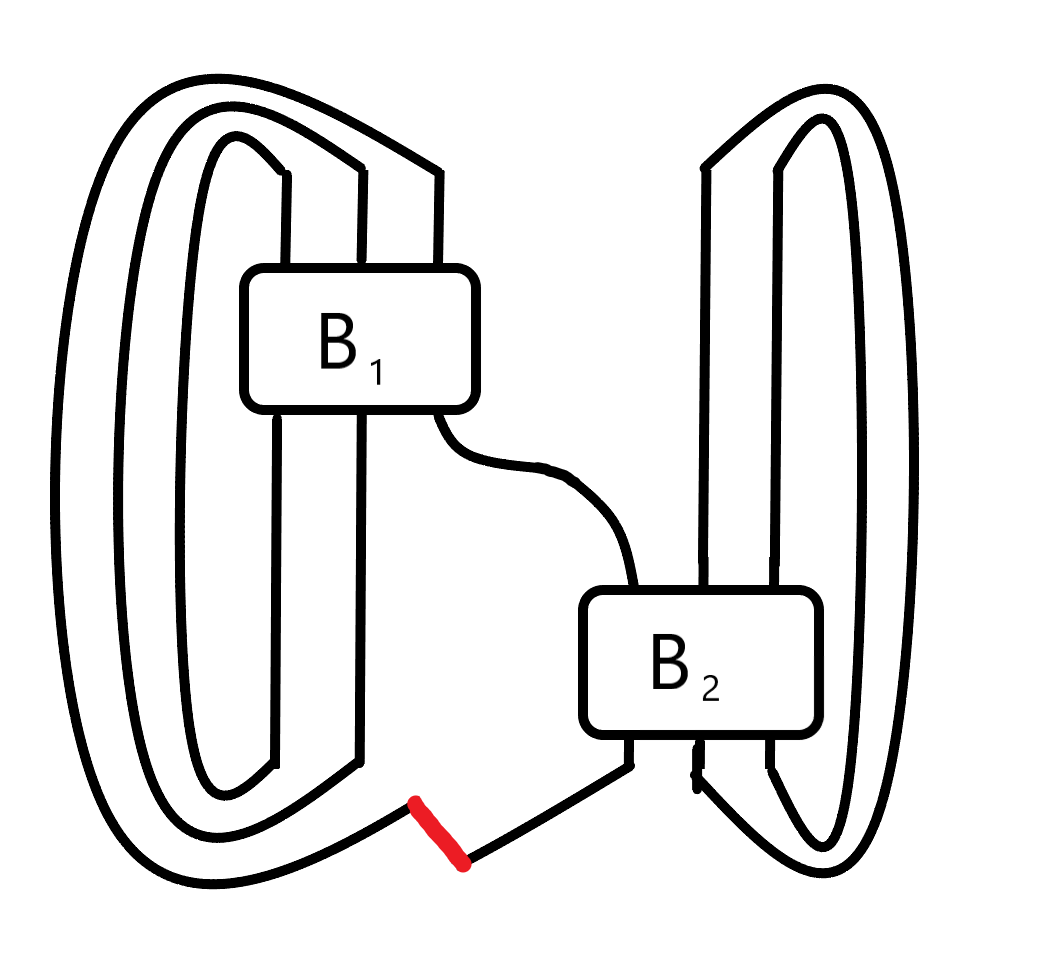}};
\end{tikzpicture}
\end{center}
\caption{A Braid with Only One $\sigma_i$}
\label{10ref}
\end{figure}

\begin{lem}
\label{gen3}
If $\hat\beta$ is a positive braid knot with unknotting number $m$, then $\hat\beta$ can be expressed as the closure of a positive braid on $2m+1$ strands or fewer.
\end{lem}
\begin{proof}
Let $n$ be the fewest number of strands on which $\hat\beta$ can be expressed as a positive braid. By \cite{Livingston}, we have $$u(\hat\beta)=m=\frac{\ell-n+1}{2},$$
implying $\ell=2m+n-1$. Since we have $n-1$ generators $\sigma_1,\cdots,\sigma_{n-1}$, Lemmas \ref{gen1} and \ref{gen2} give us that each of these generators must appear at least twice. Therefore, the length of our braid word must be at least twice the number of generators, so $\ell\ge 2(n-1)$. Substituting, we get $2m+n-1\ge 2(n-1)$, implying $n\le 2m+1$. This tell us that $2m+1$ is greater than or equal to the fewest number of strands on which $\hat\beta$ can be expressed as a positive braid.
\end{proof}

\begin{thm}
\label{finite}
For each positive integer $m$, there are a finite number of positive braid knots with unknotting number $m$.
\end{thm}
\begin{proof}
Let $\hat\beta$ be a positive braid knot with unknotting number $m$. By Lemma \ref{gen3}, we know that $brd(\hat\beta) \leq 2m+1$. This means that $\hat\beta$ can be expressed on $2m+1$ strands, as any knot with a lower index can simply add each generator between its index and $2m$ inclusive to the end of its braid word to get a braid on $2m+1$ strands with an isotopic closure. If a positive braid with unknotting number $m$ is expressed on $2m+1$ strands, by the unknotting number formula for positive braids, $\ell=2m+n-1=2m+(2m+1)-1=4m$. So we have $4m$ letters in our braid word, each of which must be one of $2m$ generators. So there are no more than $(2m)^{4m}$ positive braid knots of unknotting number $m$.
\end{proof}

Naturally, most of these $(2m)^{4m}$ possible braid words have unknotting number less than $m$, are links, or form isotopic closures, so the actual number of positive braid knots of unknotting number $m$ is drastically lower. The goal here was simply to set some finite bound on the number of positive braid knots of a given unknotting number. This is not possible for knots in general, it is known that that there are an infinite number of knots of any unknotting number greater than zero. \cite{KnotBook}.

\begin{con}
If $\hat \beta_1$ and $\hat \beta_2$ are positive braid knots such that $\hat \beta_1 \leq_g \hat \beta_2$, then there exists a positive path from $\hat\beta_2$ to $\hat\beta_1$.
\end{con}

This conjecture is true for all positive braid knots we show in this paper to be Gordian adjacent. It is also a statement that shows the value of Theorem \ref{finite}. If we have two positive braid knots known to be Gordian adjacent, we only need to look at a finite number of knots to determine whether a positive path from one to the other exists. If the conjecture is true, then there is only a finite number of possible positive paths to check before we can determine that some positive braid knot is not Gordian adjacent to another.


\printbibliography

\end{document}